\documentclass[11pt,reqno]{amsart}

\usepackage{amsmath,amssymb,mathrsfs}
\usepackage{graphicx,cite,times}
\usepackage{cases}
 \usepackage{dsfont}
\newtheorem{theo}{Theorem}[section]
\newtheorem{coro}[theo]{Corollary}
\newtheorem{lemm}[theo]{Lemma}
\newtheorem{prop}[theo]{Proposition}
\newtheorem{rema}[theo]{Remark}
\newtheorem{defi}[theo]{Definition}
\numberwithin{equation}{section}

\begin{document}

\title[Periodic Solutions of  Nonlinear Beam Equations ]{ The  existence of
 Periodic Solutions  for   Nonlinear Beam Equations on $\mathbb T^d$  by a Para-differential Method }

\author{Bochao Chen}
\address{School of
Mathematics and Statistics, Center for Mathematics and
Interdisciplinary Sciences, Northeast Normal University, Changchun, Jilin 130024, P.R.China}
\email{chenbc758@nenu.edu.cn}

\author{Yong Li}
\address{School of
Mathematics and Statistics, Center for Mathematics and
Interdisciplinary Sciences, Northeast Normal University, Changchun, Jilin 130024, P.R.China}
\email{yongli@nenu.edu.cn}

\author{Yixian Gao}
\address{School of
Mathematics and Statistics, Center for Mathematics and
Interdisciplinary Sciences, Northeast Normal University, Changchun, Jilin 130024, P.R.China}
\email{gaoyx643@nenu.edu.cn}

\thanks{The research  of YL was supported in part by NSFC grant 11571065, 11171132  and  National Research Program of China Grant 2013CB834100, The research of YG was supported in part by NSFC grant
11671071, JLSTDP 20160520094JH  and the Fundamental Research Funds for the Central Universities 2412017FZ005.}


\begin{abstract}
This paper focuses on the  construction of periodic solutions of
 nonlinear beam equations on the $d$-dimensional tori.
 For a large set of frequencies, we demonstrate that  an equivalent
form of the nonlinear equations can be obtained  by a para-differential conjugation.  Given the non-resonant  conditions on each finite dimensional subspaces, it is shown that the periodic
solutions can be constructed for the block diagonal
equation  by a classical iteration scheme.

\end{abstract}

\maketitle

\section{Introduction}
\setcounter{equation}{0}
In this paper, we study  the existence of the time-periodic solutions for the following nonlinear beam equations
\begin{equation}\label{E1.1.1}
(\partial_{tt}+\Delta ^2+m)u=\epsilon \frac{\partial F}{\partial {u}}(\omega t, x, u, \epsilon)+\epsilon f(\omega t, x), ~t\in \mathbb {R},~x\in\mathbb{T}^d=(\mathbb{R}/2\pi\mathbb{Z})^d,
\end{equation}
where $\omega>0$, $m>0$, $\epsilon\in[0,1]$,  $f$ is $
{2\pi}$-periodic in $t$  and smooth function on $\mathbb{R}\times\mathbb{T}^d$ with value in $\mathbb{R}$.
 The nonlinearity term  $F$ is also $2\pi$-periodic in time and satisfies
\begin{equation}\label{E1.2}
F(t,x, z,\epsilon) \in C^{\infty}(\mathbb{R}\times \mathbb{T}^d\times \mathbb{R}\times [0,1];\mathbb{R}), \quad
\partial^k_{z}F(t, x, 0, \epsilon)\equiv0 \quad \mathrm{ for}~ k\leq2.
\end{equation}
By means of a para-differential method together with a classical iteration scheme,
which   introduced  by Delort in \cite{Delort2011}, we try to show that there exists a small $\omega$-measure set $\mathcal O \subset [1,2]\times (0,1]$, such that for $(\omega,\epsilon)\notin\mathcal O$ with $\epsilon $ small enough,  Eq.\eqref{E1.1.1} admits a family of time-periodic solutions.

The search for periodic solutions of nonlinear PDEs has a long standing tradition. The problem has received high attention thanks to the pioneering work of Rabinowitz \cite{Rabinowitz1967periodic,rabinowitz1968periodic}.
He rephrased the problem as a variational
problem and proved the existence of periodic solutions whenever the time period $T$ is a rational
multiple of the length of  spatial interval, and the nonlinearity $f$ is  monotonic in $u$. Subsequently many related results have been obtained by Bahri, Br\'{e}zis, Corn, Nirenberg etc., see \cite{bahri1980periodic,brezis1983periodic,brezis1981periodic,brezis1978forced}, while some  recent papers can be found in \cite{ambrose2010computation,colliander2010transfer,ji2011time}. Among  most of these results,
period $T$ was required to be a rational multiple of $\pi$ (length of spatial interval). Otherwise, it  results in the ``small divisor" problem. For example,  the spectrum of the wave operator $\omega\partial_{tt}-\partial_{xx}$ approaches to zero for almost every $\omega\in\mathbb{R}\setminus \mathbb{Q}$.
In the later of 1980's, a  approach
via the KAM  method  was developed from the
viewpoint of infinite-dimensional Hamiltonian partial differential equations by Kuksin \cite{kuksin1987hamiltonian}, Eliasson\cite{eliasson1988perturbations} and Wayne
\cite{wayne1990periodic}. This method allowed one to obtain solutions whose periods are irrational multiples
of the length of the spatial interval, and it is also easily extended to construct
quasi-periodic solutions, see \cite{baldi2014kam,berti2014kam,chierchia2000kam,eliasson2010kam,gao2009invariant,geng2006kam,kuksin1556nearly}.
Later, in
\cite{bourgain1994construction,bourgain1998quasi,Bourgain1999,bourgain2000diffusion,bourgain2005green,craig1993newton} Craig,
Wayne and Bourgain retrieved the  Nash-Moser iteration method together with the Lyapunov-Schmidt
reduction which involves the Green's function analysis and the
control of the inverse of infinite matrices with  small eigenvalues,
successfully constructed the periodic and quasi-periodic solutions
of partial differential equations with Dirichlet boundary conditions
or periodic boundary conditions.  Some recent results about  Nash-Moser theorems can be found in \cite{baldi2008forced,baldi2013periodic,Berti2013Quasi} and the reference there in.
Finding periodic and quasi-periodic solutions for PDEs in higher space dimensions
 is much harder than in the one-dimensional case, mainly due to the high
degeneracy of the frequencies of the linearized equation.
For the  high dimensional beam equations with the  real-analytic nonlinearities  depending on the space
variable, in \cite{gentile2009periodic} Gentile and Procesi proved the existence of Gevrey smooth periodic solutions. Their approach is
 based on a standard Lyapunov-Schmidt decomposition,
which separates the original PDEs into two equations, traditionally called the $P$ and
$Q$ equations - combined with renormalized
 expansions {\it \.{a} la} Lindstedt to handle the
small divisor problem. In \cite{gengyou2006kam,geng2006kam}, in high dimension, for a class of PDEs with
periodic boundary conditions and with nonlocal smooth nonlinearities,
 Geng and You gave the existence of quasi-periodic solution by establishing an infinite dimensional KAM
theorem.
 Especially, the nonlinear beam equations have received much
attention by the mathematical communities, it mainly focuses on the KAM method \cite{Xu2009,Liang2006,Geng2006,Chang2015,WangandSi2012}  and Nash-Moser iteration method \cite{Shi2016,chen2017quasi}.
In addition to the above two method,
para-differential approach is also an useful tool to construct the periodic solutions of nonlinear PDEs, which  only makes use of ``symbolic calculus'' properties.
We refer to \cite{Delort2011} and \cite{chenconstruction}
for the Sch\"{o}rdinger equation and the wave equation respectively.  The properties of the operator in this paper is different from the  ones  in \cite{Delort2011,chenconstruction}
and more  difficulties appears in diagonalization of the equation.

In the present paper, we employ the para-differential approach instead of
 Nash-Moser theorems and KAM methods.
   In a Nash-Moser iteration scheme, ones have to consider the treatment of losses of derivative coming from small divisors and the convergence of the sequence of approximations at the same time. However, using para-differential  approach,
  such losses of derivative coming from small divisors will be compensated by the smoothing properties of the operator in the right hand side of the equation, as a result we don't worry about the convergence of the sequence of approximations of the solution when we treat small divisors. Furthermore, the regularity of the nonlinearity does not need to be analytic, and can depend on  space and time variables.
Since the $L^2(\mathbb{T}^d)$ can be decomposed to  the direct sum of subspace ${\rm Range}(\widetilde{\Pi}_\alpha)$ for all $\alpha\in\mathcal{A}$ defined by \eqref{E2.2.9}, one  advantage  in this article is that we solve the equation on ${\rm Range}(\widetilde{\Pi}_\alpha)$ which is a finite dimensional subspace. We just have to give a non-resonant condition on every ${\rm Range}(\widetilde{\Pi}_\alpha)$, which is different with the non-resonant condition given by the KAM method or  Nash-Moser iteration method.


\subsection{Main results}\label{subsec:2.2}
Denote by $\mathcal{D}'(\mathbb{T}\times\mathbb{T}^d)$ the space of generalized functions on
$\mathbb{T}\times\mathbb{T}^d$. To fix ideas, we shall take $\omega$ inside a fixed compact sub-interval of $(0,\infty)$, such as $\omega\in [1,2]$ (in fact
any compact interval $[a, b] \subset (0, \infty)$ is also true). After a time rescaling $t\rightarrow\frac{t}{\omega}$, we prove the existence of $2\pi$-periodic solutions in time of
\begin{equation}\label{E2.1.2}
(\omega^2\partial_{tt}+\Delta^2+m)u=\epsilon\frac{\partial F}{\partial {u}}( t,x,u,\epsilon)+\epsilon f(t,x).
\end{equation}
    Define the Sobolev space $\tilde{\mathcal{H}}^\sigma$ with $\sigma\in\mathbb{R}$ as follows:
\begin{equation}\label{E2.1.3}
\tilde{\mathcal{H}}^\sigma:=\tilde{\mathcal{H}}^\sigma(\mathbb{T}\times\mathbb{T}^d; \mathbb{R}):=\left\{u\in\mathcal{D}'(\mathbb{T}\times\mathbb{T}^d);~\|u\|^2_{\tilde{\mathcal{H}}^\sigma}<+\infty~ \text{ with}~\bar{{u}}_{j,n}={u}_{-j,-n}\right\},
\end{equation}
where, $\forall j\in\mathbb{Z}$, $\forall n\in(n_1,\cdots,n_d) \in \mathbb {Z}^d$
\begin{align}
\|u\|^2_{\tilde{\mathcal{H}}^\sigma}
&:=\sum\limits_{j\in\mathbb{Z}}
\sum\limits_{n\in\mathbb{Z}^d}
(1+j^2+|n|^4)^\sigma|\hat{u}(j,n)|^2,\label{E1.4}\\
\hat{u}(j,n)&=\frac{1}{(2\pi)^{\frac{d+1}{2}}}\int_{\mathbb{T}\times\mathbb{T}^d}e^{-\mathrm{i}j t-\mathrm{i}n\cdot x}u(t,x)~\mathrm{d}t\mathrm{d}x,\quad
|n|=\sqrt {n_1^2+\cdots+n_d^2}.\nonumber
\end{align}
 In \cite{Bourgain1999}, Bourgain gave the geometric properties of the spectrum of operator $-\Delta$ on $\mathbb{T}^d$ (The proof see $\mathrm{Lemma}~19.10$ in \cite{bourgain2005green}).
 \begin{lemm}{\bf( Bourgain)}\label{lemma1.1}
Set $\mathbb{N}^{+}:=\{d\in\mathbb{N},d>0\}$. For all $0<\beta<\frac{1}{10}$, there exist $\theta>0$ and a partition
\begin{equation*}
\mathbb{Z}^d=\bigcup\limits_{\alpha\in\mathcal{A}}\Omega_\alpha \quad\mathrm{with}\quad d\in\mathbb{N}^{+}
\end{equation*}
such that, for all $\alpha\in\mathcal{A}$, the following properties holds:
\begin{align}
&\forall~n,n'\in\Omega_\alpha,
|n-n'|+\left||n|^2-|n'|^2\right|<\theta+|n|^\beta,\label{E2.2.2}\\
\forall~n\in&\Omega_{\alpha},\forall~n'\in\Omega_{\alpha'}, ~\mathrm{with}~\alpha\neq\alpha',~|n-n'|+\left||n|^2-|n'|^2\right|>|n|^\rho,\label{E2.4}
\end{align}
where $0<\rho=\rho(\beta,d)<\beta$.
\end{lemm}
Put $\langle n\rangle:=(1+|n|^2)^{1/2}$. Using formulae \eqref{E2.2.2} and \eqref{E2.4}, we derive that for all $n\in\Omega_\alpha$
\begin{equation}\label{E2.2.7}
\Theta_0^{-1}\langle n(\alpha)\rangle\leq\langle n\rangle\leq\Theta_0\langle n(\alpha)\rangle
\end{equation}
for some $\Theta_0>0$. Moreover
\begin{equation}\label{E2.2.8}
|\Omega_\alpha|\leq\Theta_1\langle n(\alpha)\rangle^{\beta d}
\end{equation}
for some $\Theta_1>0$. We remark that $ n(\alpha)$ is a fixed constant in $\Omega_\alpha$. Let us verify the ``separation property"  of the spectrum. The spectrum of operator $\sqrt{\Delta^2+m}$ is
\begin{equation*}
\lambda_{n}=\sqrt{|n|^4+m}, \quad ~n\in\mathbb{Z}^d.
\end{equation*}
Then for any $n\in\Omega_{\alpha}$, $n'\in\Omega_{\alpha'}$, with $\alpha\neq \alpha',~\alpha,\alpha'\in\mathcal{A}$, we have
\begin{align}
\left|\lambda_{n}-\lambda_{n'}\right|=&\left|\sqrt{|n|^4+m}-\sqrt{|n'|^4+m}\right|
=\frac{|(|n|^2+|n'|^2)(|n|^2-|n'|^2)|}{\sqrt{|n|^4+m}+\sqrt{|n'|^4+m}}\nonumber\\
>&
 \begin{cases}
  \frac{1}{2\sqrt{2}}\left||n|^2-|n'|^2\right|, \quad &\text {when}~0<m\leq1,\\
  \frac{1}{2\sqrt{2m}}\left||n|^2-|n'|^2\right|, \quad & \text {when}~m>1.\label{e1.3}
 \end{cases}
\end{align}
For $n\in\mathbb{Z}^d$, let $\Pi_n$ denote a spectral projector
\begin{equation*}
\Pi_n u=\hat{u}(t,n)\frac{e^{\mathrm{i}n\cdot x}}{(2\pi)^\frac{d}{2}}=
\sum_{j\in\mathbb{Z}}^{}\hat{u}(j,n)\frac{e^{\mathrm{i}j t+\mathrm{i}n\cdot x}}{(2\pi)^\frac{d+1}{2}},\quad u\in\mathcal{D}'(\mathbb{T}\times\mathbb{T}^d),
\end{equation*}
where $t$ is considered as a parameter. For all $\alpha\in\mathcal{A}$,  set
\begin{equation}\label{E2.2.9}
\begin{array}{ll}
\widetilde{\Pi}_\alpha={\sum\limits_{n\in\Omega_\alpha}}\Pi_n.
\end{array}
\end{equation}
Define a closed subspace $\mathcal{H}^\sigma$ of $\tilde{\mathcal{H}}^\sigma$ by
\begin{align}
\mathcal{H}^\sigma:=\bigcap\limits_{\alpha\in\mathcal{A}}&
\Big\{u\in\tilde{\mathcal{H}}^\sigma(\mathbb{T}\times\mathbb{T}^d;\mathbb{R});
 ~\forall~n\in\Omega_\alpha,~\forall~j~{\text{with}}~\nonumber\\
&|j|>K_0 \langle n(\alpha)\rangle^2~~{\text{or}}~|j|<K^{-1}_0 \langle n(\alpha)\rangle^2,\hat{u}(j,n)=0\Big\},\label{E2.2.3}
\end{align}
where $K_0:=K_0(m)$ is taken in Section \ref{sec:2}. The definition of $\mathcal{H}^\sigma$ implies that, for $u\in\mathcal{H}^\sigma$, non vanishing $\hat{u}(j,n)$ have to satisfy $K^{-1}_0\langle n(\alpha)\rangle^2\leq|j|\leq K_0\langle n(\alpha)\rangle^2$ for $n\in\Omega_\alpha$. This shows that the $\tilde{\mathcal{H}}^\sigma$-norm (see \eqref{E1.4}) restricted in $\mathcal{H}^\sigma$  is equivalent to the following norms
\begin{align}\label{E2.2.6}
\left(\sum\limits_{j\in\mathbb{Z},n\in\mathbb{Z}^d}\langle n\rangle^{4\sigma}|\hat{u}(j,n)|^2\right)^{\frac{1}{2}},~\left(\sum\limits_{n\in\mathbb{Z}^d}\langle n\rangle^{4\sigma}\|\Pi_nu\|^2_{\mathcal{H}^0(\mathbb{T}\times\mathbb{T}^d;\mathbb{R})}\right)^{\frac{1}{2}}, \left(\sum\limits_{\alpha\in\mathcal{A}}\langle n(\alpha)\rangle^{4\sigma}\|\widetilde{\Pi}_\alpha u\|^2_{\mathcal{H}^0(\mathbb{T}\times\mathbb{T}^d;\mathbb{R})}\right)^{\frac{1}{2}}.
\end{align}

Our aim of this paper is to prove the following theorem.
\begin{theo}\label{theorem2.1.1}
Fix $m>0$. For some $s_0>0$, $\zeta>\beta d+\frac{d}{2}+2$ ($\beta$ is defined in Lemma \ref{lemma1.1}), $q_0>0$, if the force term  $f\in\tilde{\mathcal{H}}^{s+\zeta}(\mathbb{T}\times\mathbb{T}^d;\mathbb{R})$ with $\|f\|_{\tilde{\mathcal{H}}^{s+\zeta}}\leq q_0$, for all $s\geq s_0$, then
 there is a constant $B>0$, a subset $\mathcal{O}\subset [1,2]\times(0,1]$ and a constant $\delta_0\in(0,1]$ small enough such that:\\
$\bullet$ for all $\delta\in(0,\delta_0]$, all $\epsilon\in[0,\delta^2]$, and all $\omega\in[1,2]$ satisfying $(\omega,\epsilon)\notin \mathcal{O}$,
Eq. \eqref{E2.1.2} admits families of solutions $u\in\tilde{\mathcal{H}}^s(\mathbb{T}\times\mathbb{T}^d;\mathbb{R})$ satisfying $\|u\|_{\tilde{\mathcal{H}}^s}\leq B\epsilon\delta^{-1}$.\\
 $\bullet$ the excluded  measure satisfies
\begin{equation}\label{E2.1.4}
meas\{\omega\in[1,2];(\omega,\epsilon)\in\mathcal{O}\}\leq B\delta.
\end{equation}
\end{theo}

\subsection{Sketch of the proof}
Section \ref{sec:2} is devoted  to perform the first reduction of the equation by applying the fixed point theorem with parameters. Then the equation on $\tilde{\mathcal{H}}^{\sigma}$ is equivalent to the one on $\mathcal{H}^{\sigma}$, where $\tilde{\mathcal{H}}^{\sigma}$, $\mathcal{H}^{\sigma}$ are, respectively, defined in \eqref{E2.1.3}, \eqref{E2.2.3}. The aim of section \ref{sec:3} is to describe the para-linearization of the equation.
We firstly define classes of convenient para-differential operators which can be used in the following;
then we para-linearize the equation, and reduce it into
\begin{equation*}
(\omega^2\partial_{tt}+\Delta^2+m+\epsilon V(u,\omega,\epsilon))u=\epsilon R(u,\omega,\epsilon)u+\epsilon{f},
\end{equation*}
 where $V$ is a para-differential operator of order zero depending on $u,\omega,\epsilon$, self-adjoint, and $R$ is a smoothing operator
 depending on $u,\omega,\epsilon$. The fifth section is the core of this paper. For a new unknown $w$, owing to a para-differential conjugation, we transform the equation on $\mathcal{H}^{\sigma}$ into a new form as follows:
\begin{equation*}
(\omega^2\partial_{tt}+\Delta^2+m+\epsilon V_{\rm D}(u,\omega,\epsilon))w=\epsilon R_1(u,\omega,\epsilon)w+\epsilon {f},
\end{equation*}
 where $V_{\rm D}$ depends on $u,\omega,\epsilon$. The operator $V_{\rm D}$ is block diagonal corresponding to an orthogonal decomposition of $L^2(\mathbb{T}^d)$, which is in a sum of finite dimensional subspaces introduced by Bourgain  \cite{Bourgain1999}. The operator $R$ is still smoothing.
In section \ref{sec:6}, our main goal is to construct the solution of the block diagonal equation by a standard iteration scheme.
Combining with the non-resonant conditions \eqref{E5.1.7}, we show that  $\omega^2\partial_{tt}+\Delta^2+m+\epsilon V_{\rm D}$ is invertible on each block when $\omega$ outside a subset. To guarantee that the measure of excluded $\omega$ remains small, we have to allow small divisors when inverting $\omega^2\partial_{tt}+\Delta^2+m+\epsilon V_{\rm D}$. While, such losses of derivatives coming from small divisors may be compensated by the smoothing operator $R$ on the right-hand side of the equation. At the same time, we can construct an approximate sequence of the solution.

\section{An equivalent formulation on $\mathcal{H}^\sigma$ }\label{sec:2}
In this section, we will apply the fixed point theorem with parameters to perform the first reduction of the equation.
For convenience, we first give  some new notations. For $\sigma\in\mathbb{R},q>0$, let $B_q(\mathcal{H}^\sigma)$ stand for the
open ball with center 0, radius $q$ in $\mathcal{H}^\sigma$.  For $\sigma_1\in\mathbb{R}$, $\sigma_2\in\mathbb{R}$, we denote by
$\mathcal{L}(\mathcal{H}^{\sigma_1},\mathcal{H}^{\sigma_2})$ the space of continuous linear operators from $\mathcal{H}^{\sigma_1}$ to $\mathcal{H}^{\sigma_2}$. Specially, $\mathcal{L}(\mathcal{H}^{\sigma_1},\mathcal{H}^{\sigma_1})$ is written as $\mathcal{L}(\mathcal{H}^{\sigma_1})$. For
$\sigma_1\in\mathbb{R}$, $\sigma_2\in\mathbb{R}$, $\sigma_3\in\mathbb{R}$, let $\mathcal{L}_2(\mathcal{H}^{\sigma_1}\times\mathcal{H}^{\sigma_2},\mathcal{H}^{\sigma_3})$ denote the space of continuous bilinear operators from
$\mathcal{H}^{\sigma_1}\times\mathcal{H}^{\sigma_2}$ to $\mathcal{H}^{\sigma_3}$. Moreover, if
$T\in\mathcal{L}(\mathcal{H}^{\sigma_1},\mathcal{H}^{\sigma_2})$, then the transport ${^tT}\in\mathcal{L}(\mathcal{H}^{\sigma_2},\mathcal{H}^{\sigma_1})$.
In addition, we have to fix some real number $\sigma_0>\frac{d}{2}+1$.
For $\sigma\geq\sigma_0$, $\tilde{\mathcal{H}}^\sigma$ is a Banach algebra with respect to multiplication of functions, i.e.
\begin{equation*}
u_1,u_2\in\tilde{\mathcal{H}}^\sigma~~\Longrightarrow~~
\|u_1u_2\|_{\tilde{\mathcal{H}}^\sigma}\leq C\|u_1\|_{\tilde{\mathcal{H}}^\sigma}\|u_2\|_{\tilde{\mathcal{H}}^\sigma}.
\end{equation*}
\subsection{Functional setting}
We now give some definitions of  function space that will be used in the following. For brevity,  denote by $\mathcal{H}^\sigma_j, j=1, 2$ any one of  the spaces $
\mathcal{H}^\sigma, \mathcal{F}^\sigma, \tilde{\mathcal{H}}^\sigma$.
\begin{defi}\label{definition3.2.1}
For any $\sigma\geq\sigma_0$ and any open subset $X$ of $\mathcal{H}^\sigma_1$, $k\in\mathbb{Z}$,
 denote the space of $C^\infty$ maps $G:X\rightarrow\mathcal{H}^{\sigma-k}_2$ by $\Phi^{\infty,k}(X,\mathcal{H}^{\sigma-k}_2)$,
such that for any $u\in X\cap\mathcal{H}^s_1$ with $s\geq\sigma$, $G(u)\in\mathcal{H}^{s-k}_2$. Furthermore, the linear map
${\rm D}_uG(u)\in\mathcal{L}(\mathcal{H}^\sigma_1,\mathcal{H}^{\sigma-k}_2)$ extends as an elements of
$\mathcal{L}(\mathcal{H}^{\sigma'}_1,\mathcal{H}^{\sigma'-k}_2)$ for any $u\in X\cap\mathcal{H}^s_1$ with $s\geq\sigma$ and any $\sigma'\in[-s,s]$.
Moreover, $u\rightarrow {\rm D}_uG(u)$ is smooth from
$X\cap\mathcal{H}^s_1$ to the preceding space. In addition,
for any $u\in X\cap\mathcal{H}^s_1$ with $s\geq\sigma$, the bilinear map
${\rm D}_u^2G(u)\in\mathcal{L}_2(\mathcal{H}^\sigma_1\times\mathcal{H}^\sigma_1,
\mathcal{H}^{\sigma-k}_2)$ extends as an elements of $\mathcal{L}_2(\mathcal{H}^{\sigma_1}_1\times\mathcal{H}^{\sigma_2}_1,
\mathcal{H}^{-\sigma_3-k}_2)$ for any $\{\sigma_1,\sigma_2,\sigma_3\}=\{\sigma', -\sigma', \max(\sigma_0,\sigma')\}$ with $\sigma'\in[0,s]$.
In the same way, $u\rightarrow {\rm D}_u^2G(u)$ is smooth from $X\cap\mathcal{H}^s_1$ to the preceding space.
\end{defi}

\begin{defi}\label{definition3.2.2}
For any $\sigma\geq\sigma_0$ and any open subset $X$ of $\mathcal{H}^\sigma_1$, $k\in\mathbb{Z}$,
 denote the space of $C^1$ functions $\Phi:X\rightarrow\mathbb{R}$ by $C^{\infty,k}(X,\mathbb{R})$, such that for any $u\in X\cap\mathcal{H}^s_1$ with $s\geq\sigma$, $\nabla_u \Phi(u)\in\mathcal{H}^{s-k}_1$ and
$u\rightarrow\nabla_u\Phi(u)$ belongs to $\Phi^{\infty,k}(X,\mathcal{H}^{\sigma-k}_1)$.
\end{defi}
\begin{rema}
For $n\in\mathbb{N}$,  denote $\mathrm{D}^n_uG(u)$ the $n$-th order Frechet derivative of $G(u)$ with respect to $u$.
\end{rema}
In the remainder of this paper, we shall consider elements $G(u,\omega,\epsilon)$, $\Phi(u,\omega,\epsilon)$ of the preceding spaces depending on
$(\omega,\epsilon)$, where $(\omega,\epsilon)$ stays in a bounded domain of $\mathbb{R}^2$. If $G,\partial_\omega G,\partial_\epsilon G$ $\big({\rm resp}.$
$\Phi,\partial_\omega \Phi,\partial_\epsilon \Phi\big)$ satisfy the conditions of Definition \ref{definition3.2.1} $\big({\rm resp}.$ Definition
\ref{definition3.2.2}$\big)$,
we shall say that $G,\Phi$ are $C^1$ in $(\omega,\epsilon)$.

The following two lemmas and a corollary are applied to analyze the properties of the functionals $\Phi_1,\Phi_2$ which are given by \eqref{E3.2.10} and \eqref{E3.2.9} respectively, and the proofs can be found in the appendix in \cite{Delort2011}.
\begin{lemm}\label{lemma6.1}
If $s>\frac{d}{2}+1$, then $\tilde{\mathcal{H}}^s(\mathbb{T}\times\mathbb{T}^d;\mathbb{C})\subset L^\infty$. Furthermore, if $F$ is a smooth function defined on $\mathbb{T}\times\mathbb{T}^d\times\mathbb{C}$ satisfying $F(t,x,0)\equiv0$, there is some continuous function $\tau\rightarrow C(\tau)$, such that for any $u\in\tilde{\mathcal{H}}^s$,~$F(\cdot,u)\in\tilde{\mathcal{H}}^s$
with $\|F(\cdot,u)\|_{\tilde{\mathcal{H}}^s}\leq C(\|u\|_{\mathcal{L}^\infty})\|u\|_{\tilde{\mathcal{H}}^s}$.
\end{lemm}

\begin{lemm}\label{lemma6.2}
If $s>\frac{d}{2}+1$, when $u\in\tilde{\mathcal{H}}^s$, $v\in\tilde{\mathcal{H}}^{\sigma'}$, then $uv\in\tilde{\mathcal{H}}^{\sigma'}$ with
$\sigma'\in[-s,s]$. Moreover, for any $\sigma\in\mathbb{R}$, any $\sigma_0>\frac{d}{2}+1$, $\tilde{\mathcal{H}}^\sigma\cdot\tilde{\mathcal{H}}^{-\sigma}\subset\tilde{\mathcal{H}}^{-\max\{\sigma,\sigma_0\}}$.
\end{lemm}

\begin{coro}\label{corollary6.1}
If $F:\mathbb{T}\times\mathbb{T}^d\times\mathbb{C}\rightarrow\mathbb{C}$  is a smooth function with $F(t,x,0)\equiv0$,
then for any $\sigma>\frac{d}{2}+1$, $u\rightarrow F(\cdot,u)$ is a smooth map from $\tilde{\mathcal{H}}^\sigma$ to $\tilde{\mathcal{H}}^\sigma$.
\end{coro}
Define the following map for all $\sigma\geq\sigma_0,\sigma'>0$
\begin{equation*}
\begin{aligned}
G:~~~~&\tilde{H}^{\sigma}\cap\tilde{H}^{\sigma'}\rightarrow \tilde{H}^{\sigma'}\\
&u\mapsto F(t,x,u,\epsilon),
\end{aligned}
\end{equation*}
where $F\in C^{\infty}(\mathbb{R}\times \mathbb{T}^d\times \mathbb{R}\times [0,1];\mathbb{R})$ satisfying \eqref{E1.2}.

\begin{lemm}\label{lem3.7}
The map $G$ is $C^2$  with respect to $u$ and satisfies for all $h\in \tilde{H}^\sigma\cap\tilde{H}^{\sigma'}$
\begin{align*}
{\rm D}_uG(u)[h]=\partial_{u}F(t,x,u,\epsilon)h, \quad {\rm D}^2_uG(u)[h,h]=\partial^2_{u}F(t,x,u,\epsilon)h^2.
\end{align*}
\end{lemm}
\begin{proof}
Corollary \ref{corollary6.1} implies that $G$ is $C^2$ respect to $u$. It follows from  the continuity property of $u\mapsto\partial_{u}F(t,x,u,\epsilon)$ that
\begin{align*}
\|F(t,x,u+h,\epsilon)
&-F(t,x,u,\epsilon)-\partial_{u}F(t,x,u,\epsilon)h\|_{\tilde{\mathcal{H}}^{\sigma'}}\\
&=\left\|h\int_0^1(\partial_{u}F(t,x,u+\tau h,\epsilon)-\partial_{u}F(t,x,u,\epsilon))~\mathrm{d}\tau\right\|_{\tilde{\mathcal{H}}^{\sigma'}}\\
&\leq C(\sigma')\|h\|_{\tilde{\mathcal{H}}^{\max{\{\sigma,\sigma'\}}}} \max_{\sigma\in[0,1]}
\left\|\partial_{u}F(t,x,u+\tau h,\epsilon)-\partial_{u}F(t,x,u,\epsilon)\right\|_{\tilde{\mathcal{H}}^{\max{\{\sigma,\sigma'\}}}}\\
&=o(\|h\|_{\tilde{\mathcal{H}}^{\max{\{\sigma,\sigma'\}}}}).
\end{align*}
 Therefore for all $h\in \tilde{H}^{\sigma}\cap \tilde{H}^{\sigma'}$, we have
\begin{equation*}
{\rm D}_uG(u)[h]=\partial_{u}F(t,x,u,\epsilon)h
\end{equation*}
and $u\mapsto {\rm D}_uG(u)$ is continuous. Furthermore, it also holds
\begin{align*}
&\partial_{u}F(t,x,u+\tau h,\epsilon)h-\partial_{u}F(t,x,u,\epsilon)h-\partial^2_{u}F(t,x,u,\epsilon)h^2\\
&=h^2\int_0^1(\partial^2_{u}F(t,x,u+\tau h,\epsilon)-\partial^2_{u}F(t,x,u,\epsilon))~\mathrm{d}\tau.
\end{align*}
Similarly, we can  obtain that $G$ is twice differentiable with respect to $u$ and $u\mapsto \mathrm{D}^2_uG(u)$ is continuous.
\end{proof}

\begin{lemm}\label{Lemma3.2.1}
Let $\sigma\geq\sigma_0$, $k\in\mathbb{N}$, $X$ and $Y$ be the open subsets of $\mathcal{H}^\sigma_1$ and  $\mathcal{H}^{\sigma+k}_2$ respectively.
If $G\in\Phi^{\infty, -k}(X,\mathcal{H}^{\sigma+k}_2)$, $\Phi\in C^{\infty, k}(Y,\mathbb{R})$, and $G(X)\subset Y$, then $\Phi\circ G\in
C^{\infty, 0}(X,\mathbb{R})$.
\end{lemm}
\begin{proof}
We restrict our attention to $u\in X\cap\mathcal{H}^s_1$ with $s\geq\sigma$,
which reads $G(u)\in Y\cap\mathcal{H}^{s+k}_2$. Definitions \ref{definition3.2.1}-\ref{definition3.2.2} indicate that
\begin{equation}\label{E3.2.1}
{\rm D}_uG(u)\in\mathcal{L}(\mathcal{H}^{\sigma'}_1,\mathcal{H}^{\sigma'+k}_2)\subset \mathcal{L}(\mathcal{H}^{\sigma'}_1,\mathcal{H}^{\sigma'}_2)~\text{ for}~|\sigma'|\leq s
\end{equation}
and that $\nabla_u\Phi(G(u))\in\mathcal{H}^s_2$ for $s\geq\sigma$.
Consequently, we have for any $\sigma'$ with $|\sigma'|\leq s$,
\begin{equation}\label{E3.2.3}
{\rm D}_u(\nabla_u\Phi(G(u)))\in\mathcal{L}(\mathcal{H}^{\sigma'+k}_2,\mathcal{H}^{\sigma'}_2).
\end{equation}
It follows from  formula \eqref{E3.2.1} together with  the fact that $\nabla_u(\Phi\circ G)(u)$ is equal to ${^t{\rm D}_uG(u)}\cdot(\nabla_u\Phi(G(u)))$, we deduce
 \begin{equation*}
 \nabla_u(\Phi\circ G)(u)\in\mathcal{H}^s_1.
  \end{equation*}
Let us check $\nabla(\Phi\circ G)\in\Phi^{\infty,0}(X,\mathcal{H}^\sigma_1)$. Write ${\rm D}_u(\nabla_u(\Phi\circ G)(u))\cdot h$ as the sum of the following two terms
\begin{subequations}
\begin{numcases}{}
\label{E3.2.5a}
{^t{\rm D}_uG(u)}\cdot(({\rm D}_u\nabla_u\Phi)(G(u))\cdot {\rm D}_uG(u)\cdot h),\\
\label{E3.2.5b}
({\rm D}_u({^t\partial_uG})(u)\cdot h)\cdot\nabla_u\Phi(G(u)).
\end{numcases}
\end{subequations}
Formulae \eqref{E3.2.1} and \eqref{E3.2.3} shows that \eqref{E3.2.5a} belongs to $\mathcal{H}^{\sigma'}_1$ with $\sigma'\in[0,s]$.
According to integrating \eqref{E3.2.5b} against $h'\in\mathcal{H}^{-\sigma'}_1$, it yields that
\begin{equation}\label{E3.2.6}
\int\left(({\rm D}_u({^t{\rm D}_uG})(u)\cdot h)\cdot\nabla\Phi_u(G(u))\right)h'~{\rm d}t {\rm d}x=\int\nabla_u\Phi(G(u))\cdot {\rm D}_u^2G(u)\cdot(h,h')~{\rm d}t {\rm d}x.
\end{equation}
Definition \ref{definition3.2.1} gives that
${\rm D}_u^2G(u)\cdot(h,h')\in\mathcal{H}^{-\max\{\sigma_0,\sigma'\}+k}_2.$
Combining this with the fact that $\nabla_u\Phi(G(u))$ is in $\mathcal{H}^s_2$ (contained in $\mathcal{H}^{\max\{\sigma_0,\sigma'\}}_2$), thus
we get that the right hand side of \eqref{E3.2.6} is a continuous linear form with  $h'\in\mathcal{H}^{-\sigma'}_1$.

Next, from integrating ${\rm D}_u^2\big(\nabla_u(\Phi\circ G)(u)\big)\cdot (h_1,h_2)$  with $(h_1,h_2)\in\mathcal{H}^{\sigma_4}_1\times\mathcal{H}^{\sigma_5}_1$ against $h_3\in\mathcal{H}^{\sigma_6}_1$, it follows that
\begin{equation}\label{E3.2.7}
\int\left({\rm D}_u^2(\nabla_u(\Phi\circ G)(u))\cdot
(h_1,h_2)\right)h_3~ {\rm d}t {\rm d}x
={\rm D}_u^2\int\left(\nabla_u\Phi(G(u))\right)\left({\rm D}_uG(u)\cdot h_3\right)~{\rm d}t {\rm d}x,
\end{equation}
 where $\{\sigma_4,\sigma_5,\sigma_6\}=\{\sigma', -\sigma', \max\{\sigma_0,\sigma'\}\}$ with $\sigma'\in[0,s]$. The right-hand side of \eqref{E3.2.7} is the sum of the following four terms
\begin{subequations}
\begin{numcases}{}
\label{E3.2.8a}
{\textstyle \int} \left(\nabla_u\Phi(G(u))\right)\left({\rm D}_u^3G(u)\cdot(h_1,h_2,h_3)\right){\rm d}t {\rm d}x,\\
\label{E3.2.8b}
{\textstyle \int}\left({\rm D}_u(\nabla_u\Phi(G(u)))\cdot h_1 \right) \left({\rm D}_u^2G(u)\cdot(h_2,h_3) \right){\rm d}t {\rm d}x,\\
\label{E3.2.8c}
{\textstyle \int}\left(({\rm D}_u\nabla_u\Phi)(G(u))\cdot {\rm D}_u^2G(u)\cdot(h_1,h_2) \right) \left({\rm D}_uG(u)\cdot h_3 \right){\rm d}t {\rm d}x,\\
\label{E3.2.8d}
{\textstyle \int}\left(({\rm D}_u^2\nabla\Phi)(G(u))\cdot({\rm D}_uG(u)\cdot h_1,{\rm D}_uG(u)\cdot h_2)\right) \left({\rm D}_uG(u)\cdot h_3 \right){\rm d}t {\rm d}x,
\end{numcases}
\end{subequations}
where $(h_1,h_2)\in\mathcal{H}^{\sigma_4}_1\times\mathcal{H}^{\sigma_5}_1$.
We just consider $h_1\in\mathcal{H}^{\sigma'}_1$, $h_2\in\mathcal{H}^{-\sigma'}_1$
and $h_3\in\mathcal{H}^{\max\{\sigma_0,\sigma'\}}_1$ with $\sigma'\in[0,s]$.
In \eqref{E3.2.8a}, since $u\rightarrow {\rm D}_u^2G(u)$ is $C^1$ on $X\cap\mathcal{H}^{\max\{\sigma_0,\sigma'\}}_1$
with values in $\mathcal{L}_2(\mathcal{H}^{\sigma'}_1\times\mathcal{H}^{-\sigma'}_1;
\mathcal{H}^{-\max\{\sigma_0,\sigma'\}+k}_2)$, we obtain
\begin{equation*}
{\rm D}_u^3G(u)\cdot(h_1,h_2,h_3)\in\mathcal{H}^{-\max\{\sigma_0,\sigma'\}+k}_2.
\end{equation*}
Combing this with $\nabla_u\Phi(G(u))\in\mathcal{H}^s_2\subset\mathcal{H}^{\max\{\sigma_0,\sigma'\}}_2$ for $s\geq\sigma'\geq0$ and $s\geq\sigma$,
the two factors in \eqref{E3.2.8a} are integrable. In \eqref{E3.2.8b}, Definitions \ref{definition3.2.1}-\ref{definition3.2.2} verify
\begin{equation*}
{\rm D}_u^2G(u)\cdot(h_2,h_3)\in\mathcal{H}^{-\max\{\sigma_0,\sigma'\}+k}_2\subset\mathcal{H}^{-\sigma'+k}_2, ~{\rm D}_u(\nabla_u\Phi(G(u)))\cdot h_1\in\mathcal{H}^{\sigma_1}_2.
\end{equation*}
Consequently, the two factors in \eqref{E3.2.8b} are integrable.
In \eqref{E3.2.8c}, formula \eqref{E3.2.1} and Definitions \ref{definition3.2.1}-\ref{definition3.2.2} lead to
\begin{equation*}
{\rm D}_uG(u)\cdot h_3\in\mathcal{H}^{-\max\{\sigma_0,\sigma'\}+k}_2,
\end{equation*}
and
\[({\rm D}_u\nabla_u\Phi)(G(u))\cdot {\rm D}_u^2G(u)\cdot(h_1,h_2)\in\mathcal{H}^{-\max\{\sigma_0,\sigma'\}+k}_2,\]
which implies that the two factors in \eqref{E3.2.8c} are integrable.
In \eqref{E3.2.8d}, from ${\rm D}_uG(u)\cdot h_1\in\mathcal{H}^{\sigma'+k}_2\subset\mathcal{H}^{\sigma'}_2$ and ${\rm D}_uG(u)\cdot h_2\in\mathcal{H}^{-\sigma'+k}_2\subset\mathcal{H}^{-\sigma'}_2$, it follows that
\begin{equation*}
({\rm D}_u^2\nabla\Phi)(G(u))\cdot({\rm D}_uG(u)\cdot h_1,{\rm D}_uG(u)\cdot h_2) \in \mathcal{H}^{-\max\{\sigma_0,\sigma'\}-k}_2.
\end{equation*}
As a result, the two factors in \eqref{E3.2.8d} are integrable. This completes the proof.
\end{proof}

\subsection{An equivalent form}
Since $u$ is real-valued,  define the functionals $\Phi_1(u,f,\omega,\epsilon)$, $\Phi_2(u,\epsilon)$ by
\begin{align}\label{E3.2.10}
&\Phi_1(u,f,\omega,\epsilon):=\frac{1}{2}\int_{\mathbb{T}\times\mathbb{T}^d}(\tilde{L}_\omega u(t,x))u(t,x)~{\rm d}t{\rm d}x+\epsilon\int_{\mathbb{T}\times\mathbb{T}^d} f(t,x)u(t,x)~{\rm d}t{\rm d}x,
\quad &~u\in\tilde{\mathcal{H}}^\sigma,\\
&\Phi_2(u,\epsilon):=\int_{\mathbb{T}\times\mathbb{T}^d}F(t,x,u(t,x),\epsilon)~{\rm d}t{\rm d}x, \quad &~u\in\tilde{\mathcal{H}}^\sigma,\nonumber
\end{align}
where
\begin{equation}\label{E3.2.9}
\tilde{L}_\omega=-(\omega^2\partial_{tt}+\Delta^2+m).
\end{equation}
Then
\begin{equation*}
\nabla_u\Phi_1(u,f,\omega,\epsilon)=\tilde{L}_\omega u+\epsilon f, \quad  \nabla_u\Phi_2(u,\epsilon)=\mathrm{D}_u F(u,\epsilon).
 \end{equation*}
It follows from  the definition of $\tilde{\mathcal{H}}^{\sigma}$ that $\tilde{L}_\omega$ is a  bounded operator from $\tilde{\mathcal{H}}^{\sigma}$ to $\tilde{\mathcal{H}}^{\sigma-2}$, which shows $\Phi_1\in C^{\infty,2}(\tilde{\mathcal{H}}^\sigma,\mathbb{R})$ for $\sigma\geq\sigma_0$. Moreover, we also deduce $\Phi_2\in C^{\infty,0}(\tilde{\mathcal{H}}^\sigma,\mathbb{R})$ for $\sigma\geq\sigma_0$ from the condition \eqref{E1.2} and Lemmas \ref{lemma6.1}-\ref{lemma6.2}, Corollary \ref{corollary6.1}.
Then Eq. \eqref{E2.1.2} may be written as
\begin{equation}\label{E3.2.11}
\nabla_u(\Phi_1(u,f,\omega,\epsilon)+\epsilon\Phi_2(u,\epsilon))=0.
\end{equation}

 According to the fact of $m>0$ and formula  \eqref{E2.2.7}, it follows that, for $n\in\Omega_\alpha$ with $\alpha\in\mathcal{A}$,
 \begin{align*}
 {|n|}^4+m\geq
 \begin{cases}
  m(|n|^4+1)=\frac{m}{4}(4|n|^4+4)\geq\frac{m}{4}(|n|^4+2|n|^2+1)\geq\frac{m}{4}\Theta_0^{-4}\langle n(\alpha)^4, \quad &\text {when}~0<m\leq1,\\
  {|n|}^4+1=\frac{1}{4}(4|n|^4+4)\geq\frac{1}{4}\Theta_0^{-4}\langle n(\alpha)^4, \quad & \text {when}~m>1,
 \end{cases}
\end{align*}
and
\begin{align*}
 {|n|}^4+m\leq
 \begin{cases}
 |n|^4+1\leq|n|^4+2|n|^2+1\leq\Theta_0^{4}\langle n(\alpha)^4, \quad &\text {when}~0<m\leq1,\\
  m({|n|}^4+1)\leq m \Theta_0^{4}\langle n(\alpha)^4, \quad & \text {when}~m>1,
 \end{cases}
\end{align*}
 In addition, denote by $\mathcal{F}^\sigma$ the orthogonal complement of $\mathcal{H}^\sigma$. Owing to the definition of $\mathcal{H}^\sigma$, if $u\in\mathcal{F}^\sigma$,  then non vanishing $\hat{u}(j,n)$ for $(j,n)\in\mathbb{Z}\times\Omega_\alpha$  satisfy $|j|>K_0 \langle n(\alpha)\rangle^2$ or $|j|<K^{-1}_0 \langle n(\alpha)\rangle^2$. On one hand, if $(j,n)\in\mathbb{Z}\times\Omega_\alpha$ with $|j|>K_0 \langle n(\alpha)\rangle^2$, then we have for $\omega\in[1,2]$
 \begin{align*}
 |-\omega^2j^2+{|n|}^4+m|&\geq j^2-(|n|^4+m)>\frac{j^2}{2}+\frac{K^2_0}{4}\langle n(\alpha)\rangle^4+\frac{K^2_0}{4}\langle n(\alpha)\rangle^4- m \Theta_0^{4}\langle n(\alpha)^4\\
 &\geq\min\{\frac{1}{2},m\Theta^4_0\}(j^2+n(\alpha)^4),
 \end{align*}
 where $K_0\geq2\sqrt{m}\Theta^2_0$. On the other hand,  the following inequality holds, for $(j,n)\in\mathbb{Z}\times\Omega_\alpha$ with $|j|<K^{-1}_0 \langle n(\alpha)\rangle^2$,
 \begin{align*}
 |-\omega^2j^2+{|n|}^4+m|&\geq(|n|^4+m)-4j^2>\frac{m}{8\Theta_0^{4}}\langle n(\alpha)\rangle^4+\frac{m}{16\Theta_0^{4}}\langle n(\alpha)\rangle^4+\frac{m}{16\Theta_0^{4}}\langle n(\alpha)\rangle^4-\frac{4}{K^{2}_0} \langle n(\alpha)\rangle^4\\
 &\geq\min\{8,\frac{m}{16\Theta_0^{4}}\}(j^2+n(\langle\alpha\rangle)^4),
 \end{align*}
  where $K_0\geq\frac{8}{\sqrt{m}}\Theta^2_0$. As a consequence, if $K_0$ is greater than or equal to $\max\{2\sqrt{m},\frac{8}{\sqrt{m}}\}\Theta^2_0$, then there exists a constant $c(m)>0$ depending on $m$ such that the eigenvalues of $\tilde{L}_\omega$ satisfy for all $\omega\in[1,2]$
\begin{equation}\label{E2.10}
|-\omega^2j^2+{|n|}^4+m|\geq c(m)(|j|^2+\langle n(\alpha)\rangle^4), \quad \forall j\in\mathbb{Z}, \forall n\in\Omega_{\alpha}~\text{with}~\alpha\in\mathcal{A}.
\end{equation}
Before reducing \eqref{E3.2.11} to an equivalent form on $\mathcal{H}^\sigma$, $u,f$ are decomposed as $u_{1}+u_{2},f_1+f_2$, respectively, where  $u,f\in\tilde{\mathcal{H}}^\sigma$, $u_{1},f_1\in\mathcal{H}^\sigma$ and $u_{2},f_2\in \mathcal{F}^\sigma$.
\begin{prop}\label{proposition3.2.1}
Set $\sigma\geq\sigma_0$, $q>0$, $f_1\in B_q(\mathcal{H}^\sigma)$, $W_{q}:=B_q(\mathcal{H}^\sigma)
\times B_q(\mathcal{F}^\sigma)$. There exist $\gamma_0\in(0,1]$ small enough,
 an element
 $(u_1,f_2)\rightarrow\Psi_2(u_1,f_2,\omega,\epsilon)$ of $C^{\infty,0}(W_{q};\mathbb{R})$ and an element $(u_1,f_2)\rightarrow G(u_1,f_2,\omega,\epsilon)$ of $\Phi^{\infty,-2}(W_{q};\mathcal{F}^{\sigma+2})$,  are $C^1$ in $(\omega,\epsilon)\in[1,2]\times[0,\gamma_0]$, such that for any given subset $\mathcal{B}\subset[1,2]\times[0,\gamma_0]$,
 the following two conditions are equivalent, i.e.
\\
{(i)}~For any $(\omega,\epsilon)\in \mathcal{B}$, the function $u=(u_1,G(u_1,f_2,\omega,\epsilon))$ satisfies
\begin{equation}\label{E3.2.12}
\tilde{L}_\omega u+\epsilon f+\epsilon\nabla_u\Phi_2(u,\epsilon)=0;
\end{equation}
{(ii)}~For any $(\omega,\epsilon)\in \mathcal{B}$, the function $u_1$ satisfies
\begin{equation}\label{E3.2.13}
\tilde{L}_\omega u_1+\epsilon f_1+\epsilon\nabla_{u_1}\psi_2(u_1,f_2,\omega,\epsilon)=0.
\end{equation}
\end{prop}
\begin{proof}
Eq. \eqref{E3.2.12} may be written as the following system
\begin{subequations}
\begin{numcases}{}
\label{E3.2.14a}
\tilde{L}_\omega u_{1}+\epsilon f_1+\epsilon\nabla_{u_{1}}\Phi_2(u_{1},u_{2},\epsilon)=0,\\
\label{E3.2.14b}
\tilde{L}_\omega u_{2}+\epsilon f_2+\epsilon\nabla_{u_{2}}\Phi_2(u_{1},u_{2},\epsilon)=0.
\end{numcases}
\end{subequations}
Formula \eqref{E2.10} reads that the restriction of $\tilde{L}_\omega$ on $\mathcal{F}^{\sigma}$ is an invertible operator from $\mathcal{F}^{\sigma}$ to $\mathcal{F}^{\sigma-2}$. Then the solution of \eqref{E3.2.14b} may be expressed in terms of the form $u_2=-\epsilon \tilde{L}^{-1}_\omega f_2+\epsilon w_2$, where
\begin{equation}\label{E3.2.15}
w_{2}=-\tilde{L}^{-1}_\omega\nabla_{u_{2}}\Phi_2(u_{1},-\epsilon \tilde{L}^{-1}_\omega f_2+\epsilon w_{2},\epsilon).
\end{equation}
For all $(u_1,h)\in B_q(\mathcal{H}^\sigma)\times B_q(\mathcal{F}^\sigma)$, all $ (\omega,\epsilon)\in[1,2]\times[0,1]$, we have
\begin{equation*}
\begin{array}{l}
\|\tilde{L}^{-1}_\omega\nabla_{u_2}\Phi_2(u_1,h,\epsilon)\|_{\mathcal{F}^{\sigma+2}}\leq \frac{q_1}{2}
\end{array}
\end{equation*}
for some constant $q_1>0$. By means of the fixed point theorem with parameters, there exists $\gamma_0\in(0,1]$, such that for any $(u_1,f_2)\in W_q$, any $\epsilon\in(0,\gamma_0]$, Eq. \eqref{E3.2.15} has a unique solution $w_{2}\in B_{q_1}(\mathcal{F}^{\sigma+2})$, which is denoted by $G(u_1,f_2,\omega,\epsilon)$. As a consequence
\begin{equation}\label{E3.16}
u_{2}=-\epsilon \tilde{L}^{-1}_\omega f_2+\epsilon G.
\end{equation}
Let us verify that $G\in\Phi^{\infty,-2}(W_q;\mathcal{F}^{\sigma+{2}})$. Formula \eqref{E3.2.15} indicates that $G$ is a smooth function of $u_1$ with $C^1$ dependence on $(\omega,\epsilon)$ and that
$G$ belongs to $\mathcal{F}^{s+{2}}$ for all$(u_1,f_2)\in W_q\cap\tilde{\mathcal{H}}^s$ with $s\geq\sigma$.
Furthermore
\begin{align*}
&{\rm D}_{u_1}G(u_1,f_2,\omega,\epsilon)=-\tilde{L}^{-1}_\omega({\rm Id}-\epsilon M_2(u_1,f_2,\omega,\epsilon)\tilde{L}^{-1}_\omega)^{-1}M_1(u_1,f_2,\omega,\epsilon),\\
&{\rm D}_{f_2}G(u_1,f_2,\omega,\epsilon)=\epsilon\tilde{L}^{-1}_\omega({\rm Id}-\epsilon M_2(u_1,f_2,\omega,\epsilon)\tilde{L}^{-1}_\omega)^{-1}M_2(u_1,f_2,\omega,\epsilon)\tilde{L}_\omega^{-1},
\end{align*}
 where
\begin{align*}
&M_1(u_1,f_2,\omega,\epsilon)=({\rm D}_{u_1}\nabla_{u_2}\Phi_2)(u_1,-\epsilon \tilde{L}^{-1}_\omega f_2+\epsilon G,\epsilon ),\\
&M_2(u_1,f_2,\omega,\epsilon)=-({\rm D}_{u_2}\nabla_{u_2}\Phi_2)(u_1,-\epsilon \tilde{L}^{-1}_\omega f_2+\epsilon G,\epsilon ).
\end{align*}
We restrict ourselves to $(u_1,f_2)\in W_q\cap\tilde{\mathcal{H}}^s$ for $s\geq\sigma$. The fact of
$\Phi_2\in C^{\infty,0}(W_q,\mathbb{R})$ gives that $M_1(u_1,f_2,\omega,\epsilon)$ $({\rm resp}.$
$M_2(u_1,f_2,\omega,\epsilon))$ sends $\mathcal{H}^{\sigma'}$ $({\rm resp}.\mathcal{F}^{\sigma'})$ to $\mathcal{F}^{\sigma'}$ for any $\sigma'\in[-s,s]$.
Choose $\gamma_0$ small enough to ensure
\begin{equation*}
 \epsilon \|M_2(u_1,f_2,\omega,\epsilon)\tilde{{L}}^{-1}_\omega\|_{\mathcal{L}(\mathcal{F}^{\sigma}
,\mathcal{F}^{\sigma})}\leq 1/2, ~{\rm for}~\epsilon\in[0,\gamma_0].
\end{equation*}
This gives rise to
\begin{equation*}
({\rm Id}-\epsilon M_2\tilde{{L}}^{-1}_\omega)^{-1}\in\mathcal{L}(\mathcal{F}^{\sigma},\mathcal{F}^{\sigma}),
\end{equation*}
which then leads to that
$\mathrm{D}_{u_1}G$  can be written as the sum of the following two terms
\begin{subequations}
\begin{numcases}{}
\label{E3.2.18a}
-\sum_{k=0}^{2N-1}\tilde{L}^{-1}_\omega(\epsilon M_2\tilde{L}^{-1}_\omega)^kM_1,\\
\label{E3.2.18b}
-\tilde{L}^{-1}_\omega(\epsilon M_2\tilde{L}^{-1}_\omega)^N({\rm Id}-\epsilon M_2\tilde{L}^{-1}_\omega)^{-1}(\epsilon M_2\tilde{L}^{-1}_\omega)^{N}M_1.
\end{numcases}
\end{subequations}
If $N$ is chosen large enough relatively to $s$, then $(\epsilon M_2\tilde{{L}}^{-1}_\omega)^{N}M_1$ sends $\mathcal{H}^{\sigma'}$ to $\mathcal{F}^{\sigma}$
for any $\sigma'\in[-s,s]$. Then \eqref{E3.2.18b} belongs to $\mathcal{F}^{s+{2}}\subset \mathcal{F}^{\sigma'+{2}}$.
Moreover, \eqref{E3.2.18a} is bounded from $\mathcal{H}^{\sigma'}$ to $\mathcal{F}^{\sigma'+{2}}$ for any $\sigma'\in[-s,s]$.
Therefore ${\rm D}_{u_1}G$ extends as an element of $\mathcal{L}(\mathcal{H}^{\sigma'},\mathcal{F}^{\sigma'+{2}})$ for any $\sigma'\in[-s,s]$. The discussion on ${\rm D}_{f_2}G$, $\mathrm{D}^2G$ is similar to the one as above and so is omitted. Clearly, $\mathrm{D} G,\mathrm{D}^2G$  are smooth with $C^1$ dependence on $(\omega,\epsilon)$.
Consequently, $G$ is in $\Phi^{\infty,-{2}}(W_q;\mathcal{F}^{\sigma+{2}})$. Owing to \eqref{E3.2.9} and \eqref{E3.2.3}, it follows that
\begin{align*}
\Phi_1(u_1,u_2,\omega,\epsilon)+\epsilon\Phi_2(u_1,u_2,\epsilon)
=&\frac{1}{2}\int(\tilde{L}_\omega u_1)u_1~{\rm d}t{\rm d}x+\epsilon\int f_1u_1~{\rm d}t{\rm d}x
\\&+\frac{1}{2}\int(\tilde{L}_\omega u_2)u_2~{\rm d}t{\rm d}x+\epsilon\int f_2u_2~{\rm d}t{\rm d}x+\epsilon \Phi_2(u_1,u_2,\epsilon).
\end{align*}
Substituting \eqref{E3.16} into the above expression,  we can get a new functional about $(u_1,f_2,\omega,\epsilon)$, which is denoted by $\Psi(u_1,f_2,\omega,\epsilon)$.
A simple calculation yields
\begin{align*}
\Psi(u_1,f_2,\omega,\epsilon)
=\frac{1}{2}&\int(\tilde{L}_\omega u_1)u_1~{\rm d}t{\rm d}x+\epsilon\int f_1u_1~{\rm d}t{\rm d}x\\
&-\frac{\epsilon^2}{2}\int(\tilde{L}^{-1}_\omega f_2)f_2~{\rm d}t{\rm d}x
+\epsilon\psi_2(u_1,f_2,\omega,\epsilon),
\end{align*}
where
\begin{align}\label{e1.2}
\psi_2(u_1,f_2,\omega,\epsilon)=\frac{\epsilon}{2}\int G(\tilde{L}_\omega G)~\mathrm{d}t\mathrm{d}x+\Phi_2(u_1,-\epsilon \tilde{L}^{-1}_\omega f_2+\epsilon G,\epsilon).
\end{align}
The first term in the right hand side of \eqref{e1.2} belongs to $C^{\infty,2}(\mathcal{F}^{\sigma+2},\mathbb{R})$ thanks to that $\tilde{L}_\omega$ is a bounded operator from $\mathcal{F}^{\sigma+2}$ to $\mathcal{F}^{\sigma}$.
 It follows from  Lemma \ref{Lemma3.2.1} that $\psi_2\in C^{\infty,0}(W_q,\mathbb{R})$. Moreover
\begin{align*}
\nabla_{u_1}\Psi(u_1,f_2,\omega,\epsilon)[h]
&=\nabla_{u_1}\Phi_0(u_1,u_2,\omega,\epsilon)[h]+{^t[{\rm D}_{u_1}u_2(u_1,f_2,\omega,\epsilon)[h]]}\cdot
\nabla_{u_2}\Phi_0(u_1,u_2,\omega,\epsilon)\\
&=\nabla_{u_1}\Phi_0(u_1,u_2(u_1,f_2,\omega,\epsilon),\omega,\epsilon)[h]\\
&=\int(\tilde{L}_\omega u_1+\epsilon f_1+\epsilon\nabla_{u_1}\psi_2(u_1,f_2,\omega,\epsilon))h~{\rm d}t{\rm d}x,
\end{align*}
where $\Phi_0:=\Phi_1+\epsilon\Phi_2$. Hence $u_1$ is a critical point of $\Psi$ if and only if it is a solution of Eq. \eqref{E3.2.13}.
\end{proof}
Proposition \ref{proposition3.2.1} gives that we just look for families of solutions $u_1\in\mathcal{H}^\sigma$ to Eq. \eqref{E3.2.13}.
To simplify this problem,  we leave out that $\psi_2$ (defined in \eqref{e1.2}) depends on the $f_2$.  Fixing the force term $f=f_1+f_2$ and putting $q>0$,  $\epsilon\in[0,\gamma_0]$ with $\gamma_0\in(0,1]$ small enough, we turn to study the following  equation
\begin{equation}\label{E3.3.4}
\tilde{L}_\omega u+\epsilon f+\epsilon \nabla_u\psi_2(u,\omega,\epsilon)=0,
\end{equation}
 where $u\in B_q(\mathcal{H}^\sigma)$, $f\in\mathcal{H}^s$, $\psi_2\in C^{\infty,0}(B_q(\mathcal{H}^\sigma),\mathbb{R})$, with $\sigma\in[\sigma_0,s]$.

\section{Para-linearization of the equation}\label{sec:3}

 Applying the equivalent norms in \eqref{E2.2.6}, the aim of this section is to reduce \eqref{E3.3.4} into a para-differential equation
  on $\mathcal{H}^{\sigma}$. We first define classes of operators.

\subsection{Spaces of operators}
Define the spaces $\tilde{\mathcal{H}}^\sigma_{\mathbb{C}}:=\tilde{\mathcal{H}}^\sigma(\mathbb{T}\times\mathbb{T}^d;\mathbb{C})$,
${\mathcal{H}}^\sigma_{\mathbb{C}}:=\mathcal{H}^\sigma(\mathbb{T}\times\mathbb{T}^d;\mathbb{C})$ for complex valued functions. Other notations
 are defined in the similar way as in section \ref{sec:2}.

\begin{defi}\label{definition3.1.1}
Set $\chi\in\mathbb{R}$, $q>0$ with $u\in B_q(\mathcal{H}^\sigma_{\mathbb{C}})$, $N\in\mathbb{N}$ and $\sigma\in\mathbb{R}$ with $\sigma\geq \sigma_0+2N+(d+1)/2$. Denote the space of maps $u\rightarrow A(u)$ defined on $B_q(\mathcal{H}^\sigma_{\mathbb{C}})$ by $\Sigma^\chi(N,\sigma,q)$ , with values in the space of linear maps from $C^\infty(\mathbb{T}\times\mathbb{T}^d;\mathbb{C})$ to $\mathcal{D}'(\mathbb{T}\times\mathbb{T}^d;\mathbb{C})$.
And there exists a constant $C>0$, such that for any $n,n'\in\mathbb{Z}^d, u\rightarrow \Pi_n A(u)\Pi_{n'}$ is smooth with values in $\mathcal{L}(\mathcal{H}^0_{\mathbb{C}})$.
And  for any $2M\in\mathbb{N}$ with $(d+1)/2\leq M\leq \sigma-\sigma_0-2N$, any $u\in B_q(\mathcal{H}^\sigma_{\mathbb{C}})$, any $j\in\mathbb{N}$,
any $w_1,...,w_j\in \mathcal{H}^\sigma_{\mathbb{C}}$, any $n,n'\in \mathbb{Z}^d$,  the following holds:\\
(i)~For $j\geq1$, it has
\begin{align}
\|\Pi_n(\mathrm{D}^j_u A(u)\cdot(w_1,...,w_j))
\Pi_{n'}\|_{\mathcal{L}(\mathcal{H}^0_{\mathbb{C}})}\leq&
C(1+|n|+|n'|)^{2\chi}\langle n-n'\rangle^{-2M}\nonumber\\
&\times\mathds{1}_{|n-n'|\leq \frac{1}{10}(|n|+|n'|)}\prod\limits_{l=1}^j \|w_l\|_{\mathcal{H}^{\sigma_0+2N+M}_{\mathbb{C}}}.\label{E3.1.1}
\end{align}
(ii)
For $j=0$,  it has
\begin{equation*}
\|\Pi_n A(u)\Pi_{n'}\|_{\mathcal{L}(\mathcal{H}^0_{\mathbb{C}})}\leq C(1+|n|+|n'|)^{2\chi} \langle n-n'\rangle^{-2M}
\mathds{1}_{|n-n'|\leq \frac{1}{10}(|n|+|n'|)}.
\end{equation*}
\end{defi}
\begin{rema}
In formula \eqref{E3.1.1}, the term $\langle n-n'\rangle^{-2M}$ reflects the available $x$-smoothness of the symbol of a pseudo-differential operator, and the term $\mathds{1}_{|n-n'|\leq \frac{1}{10}(|n|+|n'|)}$ reflects the cut-off.
\end{rema}
\begin{rema}\label{remark3.1.1}
Definition \ref{definition3.1.1} indicates that $\partial_{tt}(A(u))$ belongs to $\Sigma^\chi(N+1,\sigma,q)$ when $A\in\Sigma^\chi(N,\sigma,q)$. In fact,
\begin{equation}\label{E2}
\partial_{tt} A (u)=\mathrm{D}_{uu} A (u)\cdot(\partial_t u)^2+\mathrm{D}_u A (u)\cdot\partial_{tt}u.
\end{equation}
If we assume $M\leq \sigma-2(N+1)-\sigma_0$, for all $u\in B_q(\mathcal{H}^\sigma_{\mathbb{C}})$, all $j\geq1$, then formulae \eqref{E3.1.1} and \eqref{E2} give that
\begin{align*}
&\|\Pi_n \mathrm{D}^j_u
(\partial_{tt}A(u))\cdot(w_1,...,w_j)\Pi_{n'}\|_{\mathcal{L}(\mathcal{H}^0_\mathbb{C})}\\
{\leq}&C(1+|n|+|n'|)^{2\chi} \langle n-n'\rangle^{-2M}\mathds{l}_{|n-n'|\leq \frac{1}{10}(|n|+|n'|)}(\|\partial_t u\|^2_{\mathcal{H}^{\sigma_0+2N+M}_{\mathbb{C}}}\\
&+
\|\partial_{tt} u\|_{\mathcal{H}^{\sigma_0+2N+M}_{\mathbb{C}}})\prod\limits_{l=1}^j \|w_l\|_{\mathcal{H}^{\sigma_0+2N+M}_{\mathbb{C}}}\\
\leq&
C_1(1+|n|+|n'|)^{2\chi} \langle n-n'\rangle^{-2M}\mathds{l}_{|n-n'|\leq \frac{1}{10}(|n|+|n'|)}(\|u\|^2_{\mathcal{H}^{\sigma}_{\mathbb{C}}}
+\|u\|_{\mathcal{H}^{\sigma}_{\mathbb{C}}})\prod\limits_{l=1}^j \|w_l\|_{\mathcal{H}^{\sigma_0+2N+M}_{\mathbb{C}}}\\
\leq&
C_2(1+|n|+|n'|)^{2\chi} \langle n-n'\rangle^{-2M}\mathds{l}_{|n-n'|\leq \frac{1}{10}(|n|+|n'|)}\prod\limits_{l=1}^j \|w_l\|_{\mathcal{H}^{\sigma_0+2(N+1)+M}_{\mathbb{C}}}
\end{align*}
 The same conclusion is reached as the case of $j=0$. This reads that $\partial_{tt}(A(u))$ is in $\Sigma^m(N+1,\sigma,q)$.
\end{rema}
\begin{rema}\label{remark3.1.3}
Owing to Definition \ref{definition3.1.1}, we have that
\begin{equation*}
\begin{array}{ll}
&\Sigma^{\chi_1}(N,\sigma,q)\subset\Sigma^{\chi_2}(N,\sigma,q)~~\text{when}~~\chi_1\leq \chi_2;\\
&\Sigma^\chi(N_1,\sigma,q)\subset\Sigma^\chi(N_2,\sigma,q)~~{\text{when}}~~N_1\leq N_2.
\end{array}
\end{equation*}
\end{rema}
\begin{lemm}\label{lemma3.1.1}
Let $\sigma,\chi,N,q$ satisfy the conditions of Definition \ref{definition3.1.1}. Then for all $u\in B_q(\mathcal{H}^\sigma_{\mathbb{C}})$, $s\in\mathbb{R}$, the operator $A(u)$ is bounded from $\mathcal{H}^s_{\mathbb{C}}$ to $\mathcal{H}^{s-\chi}_{\mathbb{C}}$ and the map $u\rightarrow A(u)$ is a smooth map from $B_q(\mathcal{H}^\sigma_{\mathbb{C}})$ to the space $\mathcal{L}(\mathcal{H}^s_{\mathbb{C}},\mathcal{H}^{s-\chi}_{\mathbb{C}})$. In addition, for all $j\in\mathbb{N}$, $u\in B_q(\mathcal{H}^\sigma_{\mathbb{C}})$, $w_1,...,w_j\in \mathcal{H}^\sigma_{\mathbb{C}}$, some constant $C>0$, the following inequality
\begin{equation}\label{E3.1.2}
\|\mathrm{D}^j_u A(u)\cdot(w_1,...,w_j)\|_{\mathcal{L}(\mathcal{H}^s_{\mathbb{C}},\mathcal{H}^{s-m}_{\mathbb{C}})}\leq C\prod\limits_{l=1}^j\|w_l\|_{\mathcal{H}^{\sigma_0+2N+(d+1)/2}_{\mathbb{C}}}.
\end{equation}
holds.
\end{lemm}

\begin{proof}
From formula \eqref{E3.1.1} with $2M=d+1$ and the second norm defined by \eqref{E2.2.6}, it follows that
\begin{equation*}
\|\mathrm{D}^j_u A(u)\cdot(w_1,...,w_j)u\|^2_{\mathcal{H}^{s-m}_{\mathbb{C}}}
{\leq}C\|u\|^2_{\mathcal{H}^s_{\mathbb{C}}}\prod\limits_{l=1}^j \|w_l\|^2_{\mathcal{H}^{\sigma_0+2N+(d+1)/2}_{\mathbb{C}}}.
\end{equation*}
This completes the proof of the lemma.
\end{proof}
\begin{defi}\label{definition3.1.2}
Put $\sigma\in\mathbb{R}$ with $\sigma\geq \sigma_0+2N+(d+1)/2,N\in\mathbb{N},\nu\in\mathbb{N},q>0,r\geq0$. One denotes by $\mathcal{R}^r_\nu(N,\sigma,q)$ the space of smooth maps $u\rightarrow R(u)$ defined on $B_q(\mathcal{H}^\sigma_{\mathbb{C}})$, with values in $ \mathcal{L}(\mathcal{H}^s_{\mathbb{C}},\mathcal{H}^{s+r}_{\mathbb{C}})$ for any $s\geq \sigma_0+\nu$, satisfying, for all $j\in\mathbb{N}$, $s\geq \sigma_0+\nu$, $u\in B_q(\mathcal{H}^\sigma_{\mathbb{C}})$, $w_1,...,w_j\in\mathcal{H}^\sigma_{\mathbb{C}}$
\begin{equation}\label{E3.1.3}
\|\mathrm{D}^j_u R(u)\cdot(w_1,...,w_j)\|_{\mathcal{L}(\mathcal{H}^s_{\mathbb{C}},\mathcal{H}^{s+r}_{\mathbb{C}})}\leq C\prod\limits_{l=1}^j\|w_l\|_{\mathcal{H}^\sigma_{\mathbb{C}}}
\end{equation}
for some constant $C>0$. When $j=0$, we have $\|R(u)\|_{\mathcal{L}(\mathcal{H}^s_{\mathbb{C}},\mathcal{H}^{s+r}_{\mathbb{C}})}\leq C.$
\end{defi}

\begin{rema}\label{remark3.1.2}
Lemma \ref{lemma3.1.1} indicates that $\Sigma^{-r}(N,\sigma,q)\subset\mathcal{R}^r_0(N,\sigma,q)$ for $r\geq 0$, $\sigma\geq\sigma_0+2N+(d+1)/2$.
\end{rema}

\begin{rema}\label{remark3.1.5}
 By means of Definition \ref{definition3.1.2}, the following holds:
 \begin{equation*}
\begin{array}{lll}
&\mathcal{R}^{r_1}_\nu(N,\sigma,q)\subset\mathcal{R}^{r_2}_\nu(N,\sigma,q)~~{\text{when}}~~r_1\geq r_2;\\
&\mathcal{R}^r_\nu(N_1,\sigma,q)\subset\mathcal{R}^r_\nu(N_2,\sigma,q)~~{\text{when}}~~N_1\leq N_2;\\
&\mathcal{R}^r_{\nu_1}(N,\sigma,q)\subset\mathcal{R}^r_{\nu_2}(N,\sigma,q)~~{\text{when}}~~\nu_1\leq \nu_2.
\end{array}
\end{equation*}
\end{rema}

\begin{prop}\label{proposition3.1.1}
{(i)}  Let $\sigma\geq\sigma_0+2N+(d+1)/2$. If $A\in\Sigma^\chi(N,\sigma,q)$, then ${A^*}\in\Sigma^\chi(N,\sigma,q)$.\\
{(ii)} Let $\chi_1,\chi_2\in\mathbb{R}$ and  assume $\sigma\geq\sigma_0+2N+(d+1)/2+\max\{\chi_1+\chi_2,0\}$. Put
\begin{equation}\label{E3.1.4}
r=\sigma-\sigma_0-2N-(d+1)/2-(\chi_1+\chi_2)\geq0.
\end{equation}
If $A\in\Sigma^{\chi_1}(N,\sigma,q)$, $B\in\Sigma^{\chi_2}(N,\sigma,q)$, then there exists $D\in\Sigma^{\chi_1+\chi_2}(N,\sigma,q)$ and $R\in \mathcal{R}^r_0(N,\sigma,q)$ such that
\begin{equation*}\label{E3.1.5}
A(u)\circ B(u)=D(u)+R(u).
\end{equation*}
\end{prop}
\begin{proof}
{\rm(i)}  It follows from Definition \ref{definition3.1.1}.\\
{\rm(ii)} Define
\begin{align*}
&D(u)=\sum_n\sum_{n'}\Pi_n\left(A(u)\circ B(u)\right)\Pi_{n'}\mathds{1}_{|n-n'|\leq \frac{1}{10}(|n|+|n'|)},\\
&R(u)=\sum_n\sum_{n'}\Pi_n\left(A(u)\circ B(u)\right)\Pi_{n'}\mathds{1}_{|n-n'|> \frac{1}{10}(|n|+|n'|)}.
\end{align*}
Since $\langle n-n'\rangle^{2M}\leq 2^{2M-1}(\langle n-k\rangle^{2M}+\langle k-n'\rangle^{2M})$ and $M\geq (d+1)/2$, we get the upper bound for $j = 0$
\begin{align*}
\|\Pi_nD(u)\Pi_{n'}\|_{\mathcal{L}(\mathcal{H}^0_{\mathbb{C}})}&\leq\sum\limits_k
\|\Pi_nA(u)\Pi_k\|_{\mathcal{L}(\mathcal{H}^0_{\mathbb{C}})}\|\Pi_kB(u)\Pi_{n'}\|_{\mathcal{L}(\mathcal{H}^0_{\mathbb{C}})}\mathds{1}_{|n-n'|\leq \frac{1}{10}(|n|+|n'|)}\\
&\stackrel{\eqref{E3.1.1}}{\leq}C(1+|n|+|n'|)^{2(\chi_1+\chi_2)}\sum\limits_k\langle n-k\rangle^{-2M}\langle k-n'\rangle^{-2M}\mathds{1}_{|n-n'|\leq \frac{1}{10}(|n|+|n'|)}\\
&=C(1+|n|+|n'|)^{2(\chi_1+\chi_2)}\langle n-n'\rangle^{-2M}\sum\limits_k\frac{\langle n-n'\rangle^{2M}\mathds{1}_{|n-n'|\leq \frac{1}{10}(|n|+|n'|)}}{\langle n-k\rangle^{2M}\langle k-n'\rangle^{2M}}\\
&\leq C_1(1+|n|+|n'|)^{2(\chi_1+\chi_2)}\langle n-n'\rangle^{-2M}\mathds{1}_{|n-n'|\leq \frac{1}{10}(|n|+|n'|)}.
\end{align*}
Similarly, the estimates of $\|\Pi_n\mathrm{D}^j_uD(u)\cdot(w_1,...,w_j)\Pi_{n'}\|_{\mathcal{L}(\mathcal{H}^0)}$ for $j\geq1$ are
obtained. Obviously, either $|n-k|\geq \frac{1}{2}(|n-n'|)$ or $|n'-k|\geq \frac{1}{2}(|n-n'|)$ holds. This gives also rise to
\[|n-n'|\leq\frac{1}{2}(|n|+|n'|).\] Based on these facts, by formula {\eqref{E3.1.1}}, we infers that, for $j = 0$,
\begin{align*}
\|\Pi_nR(u)\Pi_{n'}\|_{\mathcal{L}(\mathcal{H}^0_{\mathbb{C}})}
{\leq}&C(1+|n|+|n'|)^{2(\chi_1+\chi_2)}\sum\limits_k\langle n-k\rangle^{-2M}\langle k-n'\rangle^{-2M}\\
&\times\mathds{1}_{|n-k|\leq \frac{1}{10}(|n|+|k|)}\mathds{1}_{|k-n'|\leq \frac{1}{10}(|k|+|n'|)}
\mathds{1}_{|n-n'|> \frac{1}{10}(|n|+|n'|)}\\
\leq&C_1(1+|n|+|n'|)^{2(\chi_1+\chi_2)-2M}\mathds{1}_{|n-n'|\leq\frac{1}{2}(|n|+|n'|)}\\
\leq&C_1(1+|n|+|n'|)^{2(\chi_1+\chi_2)-2M+(d+1)}\langle n-n'\rangle^{-(d+1)}\mathds{1}_{|n-n'|\leq\frac{1}{2}(|n|+|n'|)},
\end{align*}
where $M:=\sigma-\sigma_0-2N$. The same argument as the proof of lemma \ref{lemma3.1.1} derives that $R(u)$ sends $\mathcal{H}^s_{\mathbb{C}}$ to $\mathcal{H}^{s+r}_{\mathbb{C}}$ for any $s$, where $r$ is given by \eqref{E3.1.4}. An reason similar to the one as $j=0$ yields
the estimates of $\|\mathrm{D}^j_u R(u)\cdot(W_1,...,W_j)\|_{\mathcal{L}(\mathcal{H}^s_{\mathbb{C}},\mathcal{H}^{s+r}_{\mathbb{C}})}$ for $j\geq1$ .
\end{proof}

In the rest of this paper, we use those operators $A(u)$ $\big({\rm resp}.$ $R(u)\big)$ of $\Sigma^m(N,\sigma,q)$ $\big({\rm resp}.$ $\mathcal{R}^r_\nu(N,\sigma,q)\big)$ sending real valued functions to real valued functions, i.e. $\overline{A(u)}=A(u)$ $\big({\rm resp}.$ $\overline{R(u)}=R(u)\big)$. Furthermore, we shall consider operators $A(u,\omega,\epsilon),R(u,\omega,\epsilon)$ depending on $(\omega,\epsilon)$, where $(\omega,\epsilon)$ stays in a bounded domain of $\mathbb{R}^2$. If $(\omega,\epsilon)\rightarrow \Pi_n A(u,\omega,\epsilon)\Pi_{n'}$ $\big({\rm resp}.$ $(\omega,\epsilon)\rightarrow R(u,\omega,\epsilon)\big)$ is $C^1$ in $(\omega,\epsilon)$ with values in $\mathcal{L}(\mathcal{H}^0)$ $\big({\rm resp}.$ $\mathcal{L}(\mathcal{H}^s,\mathcal{H}^{s+r})\big)$ and if $\partial_\omega A,\partial_\epsilon A$ $\big({\rm resp}.$ $\partial_\omega R,\partial_\epsilon R\big)$ satisfy \eqref{E3.1.1} $\big({\rm resp}.$ \eqref{E3.1.3}\big), then we shall say that operators $A(u,\omega,\epsilon)$ $\big({\rm resp}.$ $R(u,\omega,\epsilon)\big)$ are $C^1$ in $(\omega,\epsilon)$.

\subsection{Reduce to a para-differential equation}
Denote by $\mathscr {R}[(X^k)^\tau]$ with $k,\tau\in\mathbb{N}^{\varrho}$ the space of polynomials composed by monomials $(X^k)^\tau$ whose weights are equal. If $(X^{k_1})^{\tau_1} \cdots (X^{k_\varrho})^{\tau_\varrho}$ is a monomial, then the weight of $Q,Q\in\mathscr {R}[(X^k)^\tau]$ is defined as ${|k_1|}\tau_1+\cdots +{|k_\varrho|}\tau_\varrho$. Remark that $k_i$ is a multiple index for $1\leq i\leq\varrho$. In addition, abusing notations, $\mathscr {R}[(X^k)^\tau]$ with $k,\tau\in\mathbb{N}^{\vartheta}$ for all $\vartheta\leq\varrho$ denotes the space composed by polynomials whose weights are less or equal to  ${|k_1|}\tau_1+\cdots +{|k_\varrho|}\tau_\varrho$. Let $U$ be an open subset of $\mathcal{H}^{\sigma_0}$, $\psi$ belong to $C^{\infty,0}(X,\mathbb{R})$. For any $u\in U\cap\mathcal{H}^{+\infty},w_1,w_2\in\mathcal{H}^{+\infty}$, we set
\begin{equation}\label{E3.3.1}
L(u;w_1,w_2)=\mathrm{D}^2_u\psi(u)\cdot(w_1,w_2).
\end{equation}
This is a continuous bilinear form in $(w_1,w_2)\in\mathcal{H}^0\times\mathcal{H}^0$. By Riesz theorem, formula \eqref{E3.3.1} can be written as
\begin{equation*}
L(u;w_1,w_2)=\int_{\mathbb{T}\times\mathbb{T}^d}(W(u)w_1)w_2~{\rm d }t {\rm d}x
\end{equation*}
for some symmetric $\mathcal{H}^0$-bounded operator $W(u)$. Definition \ref{definition3.2.2} infers that $u\rightarrow \mathrm{D}^2_u\psi(u)$ is a smooth map defined on $U$ with values in the space of continuous bilinear forms on $\mathcal{H}^0\times\mathcal{H}^0$. This shows that $u\rightarrow W(u)$ is smooth with values in $\mathcal{L}(\mathcal{H}^0,\mathcal{H}^0)$, which then gives for any $u\in U\cap\mathcal{H}^{+\infty},w_1,w_2\in\mathcal{H}^{+\infty}$
\begin{align}
L(u;\partial_{x_j}w_1,w_2)+L(u;w_1,\partial_{x_j}w_2)&=\int_{\mathbb{T}\times\mathbb{T}^d}(\partial_{x_j}W(u)w_1)w_2~{\rm d }t {\rm d}x\nonumber\\
&=-(\mathrm{D}_uL)(u;w_1,w_2)\cdot(\partial_{x_j}u).\label{E3.3.2}
\end{align}
\begin{lemm}\label{lemma3.3.1}
Let $q>0$. For $l\in\mathbb{N},N\in\mathbb{N},N'\in\mathbb{N}$, there are  polynomials $Q^l_N\in\mathscr {R}[(X^k)^\tau]$,
of weight less or equal to $N$, a constant $C>0$, depending only on $l,q,N'$,
such that for all $u\in B_q(\mathcal{H}^{\sigma_0})\cap U\cap\mathcal{H}^{+\infty}$, $h_1,\cdots,h_l\in\mathcal{H}^{+\infty}$, $ n,n'\in\mathbb{Z}^d$, the following holds:
\begin{align}\label{E3.3.3}
\|\Pi_n\mathrm{D}^l_uW(u)\cdot(h_1,\cdots,h_l)\Pi_{n'}\|_{\mathcal{L}(\mathcal{H}^0)}\leq
C\langle n-n'\rangle^{-N'}
\sum\limits_{N_0+\cdot+N_l=N'}Q^l_{N_0}((\|\partial^ku\|_{\mathcal{H}^{\sigma_0}})^\tau)\prod\limits_{l'=1}^{l}\|h_{l'}\|_{\mathcal{H}^{\sigma_0+\frac{N_{l'}}{2}}},
\end{align}
where $Q^l_{N_0}((\|\partial^ku\|_{\mathcal{H}^{\sigma_0}})^\tau)$ is the polynomial composed by these monomials like
\[(\|\partial^{k_1}u\|_{\mathcal{H}^{\sigma_0}})^{\tau_1}
\cdots(\|\partial^{k_\vartheta}u\|_{\mathcal{H}^{\sigma_0}})^{\tau_\vartheta}\] with ${|k_1|}\tau_1+\cdots+{|k_\vartheta|}\tau_\vartheta\leq N_0$.
\end{lemm}

\begin{proof}
Using $^t\Pi_n=\Pi_{-n}$ and \eqref{E3.3.2},
for $l=0$, we deduce for any $u\in B_q(\mathcal{H}^{\sigma_0})\cap U\cap\mathcal{H}^{+\infty}$, any $w_1,w_2\in\mathcal{H}^{+\infty}$
\begin{align}
(n_{j}-n'_{j})\int\big(\Pi_nW(u)\Pi_{n'}w_1\big)w_2~{\rm d}t{\rm d}x
&=(n_{j}-n'_{j})\int\big(W(u)\Pi_{n'}w_1\big)\Pi_{-n}w_2~{\rm d}t{\rm d}x\nonumber\\
&=\mathrm{i}(L(u;\partial_{x_j}\Pi_{n'}w_1,\Pi_{-n}w_2)+L(u;\Pi_{n'}w_1,\partial_{x_j}\Pi_{-n}w_2))\nonumber\\
&=-{\rm i}(\mathrm{D}_uL)(u;\Pi_{n'}w_1,\Pi_{-n}w_2)\cdot(\partial_{x_{j}}u).\label{E1.5}
\end{align}
Moreover, we have
\begin{equation*}
\langle n-n'\rangle=\sqrt{1+(n_1-n'_1)^2+\cdots+(n_d-n'_d)^2}\leq1+|n_1-n'_1|+\cdots+|n_d-n'_d|.
\end{equation*}
Iterating the above computation \eqref{E1.5}, it yields that
\begin{equation*}
\langle n-n'\rangle^{N'}\left|\int(\Pi_nW(u)\Pi_{n'}w_1)w_2~{\rm d}t {\rm d}x \right|
\end{equation*}
is bounded from above by finite sum of
\begin{equation}\label{E200}
|(\mathrm{D}^\vartheta_uL)(u;\Pi_{n'}w_1,\Pi_{-n}w_2)(\partial^{\kappa_1}u,
\cdots,\partial^{\kappa_\vartheta}u)|
\end{equation}
with ${|\kappa_1|}+\cdots+{|\kappa_\vartheta|}\leq N'$. According to the properties of the operator $L$, the term in \eqref{E200} is bounded from above by
\begin{equation*}
\begin{array}{l}
C\|\Pi_{n'}w_1\|_{\mathcal{H}^0}\|\Pi_{-n}w_2\|_{\mathcal{H}^0}
\prod\limits_{j=1}^\vartheta\|\partial^{\kappa_{j}}u\|_{\mathcal{H}^{\sigma_0}}.
\end{array}
\end{equation*}
The remainder of the discussion on $l\geq1$ is analogous to the case of $l=0$, we have \eqref{E3.3.3} for any $l\geq0$.
\end{proof}
Put for $p\in\mathbb{N}$, $u\in\mathcal{H}^0$
\begin{equation}\label{E3.3.5}
\begin{array}{ll}
&\Delta_pu=\sum\limits_{\stackrel{n\in\mathbb{Z}^d}{2^{p-1}\leq|n|<2^p}}\Pi_{n}u, p\geq1, \quad  \Delta_0u=\Pi_0u,\\
&S_pu=\sum\limits_{p'=0}^{p-1}\Delta_{p'}u=\sum\limits_{n\in\mathbb{Z}^d,|n|<2^{p-1}}\Pi_{n}u, p\geq1, \quad S_0u=0.
\end{array}
\end{equation}
\begin{lemm}\label{lemma3.3.2}
Let $q>0,\sigma\in\mathbb{R}$ with $\sigma\geq\sigma_0+(d+1)/2$, $\gamma_0\in(0,1]$ with $\gamma_0$ small enough.
There exists a map $(u,\omega,\epsilon)\rightarrow W(u,\omega,\epsilon)$ on $B_q(\mathcal{H}^\sigma)\times[1,2]\times[0,\gamma_0]$
with values in $\mathcal{L}(\mathcal{H}^0)$, which is symmetric and is $C^\infty$ in $u$ with $C^1$ in $(\omega,\epsilon)$, such that for any $(u,\omega,\epsilon)$
\begin{equation}\label{E3.3.7}
\psi_2(u,\omega,\epsilon)=\int_{\mathbb{T}\times\mathbb{T}^d}\big(W(u,\omega,\epsilon)u\big)u~{\rm d}t {\rm d}x.
\end{equation}
Furthermore, fix some $l\in\mathbb{N}$, $N\in\mathbb{N}$, $N'\in\mathbb{N}$. There exists polynomials $Q^l_N\in\mathscr {R}[(X^k)^\tau]$,
of weight equal to $N$, and a constant $C$, depending on $l,q,N'$, such that, for all $u\in B_q(\mathcal{H}^\sigma)$, $\epsilon\in[0,\gamma_0]$, $\omega\in[1,2]$, $(\eta_1,\eta_2)\in\mathbb{N}^2$ with $\eta_1+\eta_2\leq1$,
$(h_1,\cdots,h_l)\in(\mathcal{H}^\sigma)^l$, $n,n'\in\mathbb{Z}^d$, the following holds:
\begin{align}
\|\Pi_n\partial^{\eta_1}_\omega\partial^{\eta_2}_\epsilon \mathrm{D}^l_uW(u,\omega,\epsilon)\cdot(h_1,
&\cdots,h_l)
\Pi_{n'}\|_{\mathcal{L}(\mathcal{H}^0)}
\leq
C\langle n-n'\rangle^{-N'} \nonumber\\
&\times\sum\limits_{N_0+\cdots+N_l=N'}Q^l_{N_0}((\|\partial^k S(n,n')u\|_{\mathcal{H}^{\sigma_0}})^\tau)
\prod\limits_{l'=1}^{l}\|S(n,n')h_{l'}\|_{\mathcal{H}^{\sigma_0+\frac{N_{l'}}{2}}}, \label{E3.3.8}
\end{align}
where
$S(n,n')=\sum\limits_{|n''|\leq2(1+\min(|n|,|n'|))}\Pi_{n''}$ for $n''\in\mathbb{Z}^d$.
\end{lemm}
\begin{proof}
Without loss of generality, we restrict our attention on that $W$ depends only on $u$. Definition \eqref{e1.2} shows that $\mathrm{D}^k_u\psi_2$ with $k\leq2$ is continuous. It can be seen that
\begin{equation*}
 S_pu\stackrel{\mathcal{H}^\sigma}{\longrightarrow} u\quad \text{as}\quad p\rightarrow +\infty
 \end{equation*}
using the definition in \eqref{E3.3.5}. Then
\begin{align*}
\psi_2(u)&=\sum\limits_{p_1=0}^{+\infty}\big(\psi_2(S_{p_1+1}u)-\psi_2(S_{p_1}u)\big)\\
&
=\sum\limits_{p_1=0}^{+\infty}\int_0^1(\mathrm{D}_u\psi_2)
(S_{p_1}u+\theta_1\Delta_{p_1}u)~\mathrm{d}\theta_1\cdot\Delta_{p_1}u\\
&
=\sum\limits_{p_1=0}^{+\infty}\sum\limits_{p_2=0}^{+\infty}\int_0^1\int_0^1(\mathrm{D}^2_u\psi_2)
\big(\Omega_{p_1,p_2}(\theta_1,\theta_2)u\big)~\mathrm{d}\theta_2\cdot\big(\Delta_{p_2}(S_{p_1}+\theta_1\Delta_{p_1})u,
\Delta_{p_1}u\big)~{\rm d}\theta_1,
\end{align*}
where $\Omega_{p_1,p_2}(\theta_1,\theta_2)=\Pi^2_{l=1}(S_{p_l}+\theta_l\Delta_{p_l})$.
According to Lemma \ref{lemma3.3.1} and discussion before it ,
there exists a symmetric operator $\widetilde{W}(\Omega_{p_1,p_2}(\theta_1,\theta_2)u)$ satisfying \eqref{E3.3.3}  such that
\begin{equation*}
\mathrm{D}^2\psi_2(\Omega_{p_1,p_2}(\theta_1,\theta_2)u)\cdot(w_1,w_2)=\int\big(\widetilde{W}(\Omega_{p_1,p_2}(\theta_1,\theta_2)u)
w_1\big)w_2~{\rm d} t {\rm d}x.
\end{equation*}
Thus we can  get \eqref{E3.3.7}, where
\begin{align*}
W(u)=&\frac{1}{2}\sum\limits_{p_1}\sum\limits_{p_2}\int_0^1\int_0^1\Delta_{p_1}\big(\widetilde{W}
(\Omega_{p_1,p_2}(\theta_1,\theta_2)u)\Delta_{p_2}(S_{p_1}+\theta_1\Delta_{p_1})\big)~{\rm d}\theta_1 {\rm d}\theta_2\\
&+\frac{1}{2}\sum\limits_{p_1}\sum\limits_{p_2}\int_0^1\int_0^1\Delta_{p_2}(S_{p_1}+\theta_1\Delta_{p_1})
\big(\widetilde{W}(\Omega_{p_1,p_2}(\theta_1,\theta_2)u)\Delta_{p_1}\big)~{\rm d}\theta_1{\rm d}\theta_2.
\end{align*}
It is clear that which $W(u)$ is a symmetric operator. The definition of $S(n,n')$ establishes  that
\begin{equation*}
\Pi_nW(u)\Pi_{n'}=\Pi_nW(S(n,n')u)\Pi_{n'}.
 \end{equation*}
 Combining this with \eqref{E3.3.3},  it leads to inequality \eqref{E3.3.8}. In addition, if $N'$ is larger or equal to $d+1$ to guarantee that $\sigma_0+N'/2\leq\sigma$,
 then $u, h_{l'}$ are in $\mathcal{H}^\sigma$. As a consequence, the right-hand side of \eqref{E3.3.8} is bounded by $C\langle n-n'\rangle^{-N'}$. This gives that $W(u)$ is bounded from $\mathcal{H}^0$ to $\mathcal{H}^0$.
\end{proof}

\begin{prop}\label{proposition3.3.1}
Let $q>0$, $\sigma\in\mathbb{R}$ with $\sigma\geq\sigma_0+(d+1)/2$, $\gamma_0\in(0,1]$ with $\gamma_0$ small enough, and
\begin{equation*}\label{E3.3.10}
r=\sigma-\sigma_0-(d+1)/2.
\end{equation*}
There is a symmetric element $\tilde{V}\in\Sigma^0(0,\sigma,q)$ and an element $\tilde{V}\in \mathcal{R}^r_0(0,\sigma,q)$, where $\tilde{V},\tilde{R}$ are also  $C^1$ in $(\omega,\epsilon)$ respectively, such that for all $u\in B_q(\mathcal{H}^\sigma)$, $\epsilon\in[0,\gamma_0]$, $\omega\in[1,2]$
\begin{equation*}\label{E3.3.11}
\nabla_u \psi_2(u,\omega,\epsilon)=\tilde{V}(u,\omega,\epsilon)u+\tilde{R}(u,\omega,\epsilon)u.
\end{equation*}
\end{prop}
\begin{proof}
For $h_1\in\mathcal{H}^{+\infty}$, it follows from
Lemma \ref{lemma3.3.2} that
\begin{equation}\label{E3.3.12}
{\rm D}_u\psi_2(u,\omega,\epsilon)\cdot h_1=2\int(W(u,\omega,\epsilon)u)h_1~{\rm d}t{\rm d}x
+\int(({\rm D}_uW(u,\omega,\epsilon)\cdot h_1)u)u~{\rm d}t{\rm d}x.
\end{equation}

On one hand, let us study the first term  in the right hand side of \eqref{E3.3.12}. Define
\begin{align*}
&\tilde{V}(u,\omega,\epsilon)=2\sum\limits_{n,n'}\mathds{1}_{|n-n'|\leq \frac{1}{10}(|n|+|n'|)}\Pi_nW(u,\omega,\epsilon)\Pi_{n'},\\
&\tilde{R}'(u,\omega,\epsilon)=2\sum\limits_{n,n'}\mathds{1}_{|n-n'|> \frac{1}{10}(|n|+|n'|)}\Pi_nW(u,\omega,\epsilon)\Pi_{n'}.
\end{align*}
With the help of formula \eqref{E3.3.8}, if $\frac{|k|}{2}\leq \frac{N'}{2}=M(\text{see}~\text{Definition}~\ref{definition3.1.1})\leq\sigma-\sigma_0$, then we have
\begin{equation*}
\|\partial^ k
S(n,n')u\|_{\mathcal{H}^{\sigma_0}}\leq C\|u\|_{\mathcal{H}^{\sigma}},\quad\|S(n,n')h_{l'}\|_{\mathcal{H}^{\sigma_0+\frac{N_{l'}}{2}}}\leq
C\|h_{l'}\|_{\mathcal{H}^{\sigma_0+M}}\leq
C\|h_{l'}\|_{\mathcal{H}^{\sigma}}
\end{equation*}
for some constant $C>0$. This reads that $\tilde{V}$ satisfies \eqref{E3.1.1}. Then
$\tilde{V}\in\Sigma^0(0,\sigma,q)$. 
Furthermore, the definition of $S(n,n')$ indicates that
\begin{equation}\label{E4.18}
\|S(n,n')w\|_{\mathcal{H}^{\sigma_0+\frac{\beta}{2}}}\leq C(1+\inf(|n|,|n'|))^{2\max\left\{\frac{\beta}{2}+\sigma_0-\sigma,0\right\}}\|w\|_{\mathcal{H}^\sigma}
\end{equation}
for some constant $C>0$. Combining this with the inequality $|n-n'|> \frac{1}{10}(|n|+|n'|)$ and formulae \eqref{E3.3.8} and \eqref{E4.18}, if $N'=2(\sigma-\sigma_0+1)$(see formula \eqref{E3.3.8}), then it yields that
\begin{align}
\|\Pi_n\partial^{\eta_1}_\omega\partial^{\eta_2}_\epsilon\mathrm{D}^l_u\tilde{R}'(u,\omega,\epsilon)\Pi_{n'}\|_{\mathcal{L}(\mathcal{H}^0)}\leq& C(1+|n|+|n'|)^{-N'}(1+\inf(|n|,|n'|))^{2\left(\frac{N'}{2}+\sigma_0-\sigma\right)}\prod\limits_{l'=1}^{l}\|h_{l'}\|_{\mathcal{H}^{\sigma}}\nonumber\\
\leq&C(1+|n|+|n'|)^{-2\left(\sigma-\sigma_0-\frac{(d+1)}{2}\right)}\langle n-n'\rangle^{-(d+1)}\prod\limits_{l'=1}^{l}\|h_{l'}\|_{\mathcal{H}^{\sigma}}.\nonumber
\end{align}
This gives that, for all $s\geq\sigma_0$, fixed $N'=2(\sigma-\sigma_0+1)$,
\begin{equation*}
\|\partial^{\eta_1}_\omega\partial^{\eta_2}_\epsilon \mathrm{D}^l_u\tilde{R}'(u,\omega,\epsilon)
\cdot(h_1,\cdots,h_l)\|_{\mathcal{L}(\mathcal{H}^s,\mathcal{H}^{s+r})}\leq C\prod\limits_{l'=1}^{l}\|h_{l'}\|_{\mathcal{H}^{\sigma}},
\end{equation*}
where $r=\sigma-\sigma_0-\frac{(d+1)}{2}$. As a consequence $\tilde{R}'\in\mathcal{R}^r_0(0,\sigma,q)$.

On the other hand, we study the second term  in the right hand side of  \eqref{E3.3.12}. For any $h,w\in\mathcal{H}^{+\infty}$, assume there exists an operator $\tilde{R}''(u,\omega,\epsilon)$ with
\begin{align*}
\int\big(({\rm D}_uW(u,\omega,\epsilon)\cdot h)u\big)w~{\rm d}t{\rm d}x=\int\big(\tilde{R}''(u,\omega,\epsilon)w\big)h~{\rm d}t{\rm d}x.
\end{align*}
From formulae \eqref{E3.3.8} and \eqref{E4.18}, if $N'\leq2(\sigma-\sigma_0)$, then it follows that, for $l=1$, $u\in B_q(\mathcal{H}^\sigma)$, $s\geq\sigma_0$,
\begin{align*}
\|\Pi_n{\rm D}_uW(u,\omega,\epsilon)
&\cdot h\Pi_{n'}\|_{\mathcal{L}(\mathcal{H}^0)}\leq C\langle n-n'\rangle^{-N'}(1+\inf(|n|,|n'|))^{2(N'/2+s+r+\sigma_0)}\|h\|_{\mathcal{H}^{-s-r}},
\end{align*}
where $w \in\mathcal{H}^s$ and $h\in\mathcal{H}^{-s-r}$. In addition, it is easy to obtain that
\begin{equation*}
 \|\Pi_{n}w\|_{\mathcal{H}^0}\leq c_n\langle
n\rangle^{-2s}\|w\|_{\mathcal{H}^s},~~\quad
 \|\Pi_{n'}u\|_{\mathcal{H}^0}\leq c'_{n'}q\langle
n'\rangle^{-2\sigma}
\end{equation*}
for $l^2$-sequences $(c_n)_n$, $(c'_{n'})_{n'}$. Decomposing $u=\Sigma_{n'}\Pi_{n'}u$ and $w=\Sigma_{n}\Pi_{n}w$, the following estimate
\begin{align}
\|\big(({\rm D}_uW(u,\omega,\epsilon)\cdot h)u\big)w\|_{\mathcal{H}^0}\leq&
\sum\limits_n\sum\limits_{n'}\|\Pi_n{\rm D}_uW(u,\omega,\epsilon)\cdot h\Pi_{n'}\|_{\mathcal{L}(\mathcal{H}^0)}
\|\Pi_{n'}u\|_{\mathcal{H}^0}\|\Pi_{n}w\|_{\mathcal{H}^0}\nonumber\\
\leq& C\langle n-n'\rangle^{-N'}\left(1+\inf(|n|,|n'|)\right)^{2(N'/2+s+r+\sigma_0)}\nonumber\\
&\times c_n c'_{n'}\langle n\rangle^{-2s}
\langle n'\rangle^{-2\sigma}\|w\|_{\mathcal{H}^s}\|h\|_{\mathcal{H}^{-s-r}}\label{E3.3.15}
\end{align}
holds.
Taking $r=\sigma-\sigma_0-\frac{N'}{2}$ with $N'=d+1$, we verify that the sum in $n,n'$ of \eqref{E3.3.15} is convergent. Then this leads to that $\tilde{R}''\in\mathcal{L}(\mathcal{H}^s,\mathcal{H}^{s+r})$. The argument similar to the case of $j=0$ to get the estimate of
 $\|\partial^{\eta_1}_\omega\partial^{\eta_2}_\epsilon \partial^l_u R_2(u,\omega,\epsilon)\cdot(h_1,\cdots,h_l)\|_{\mathcal{L}(\mathcal{H}^s,
 \mathcal{H}^{s+r})}$. Therefore we get $\tilde{R}''\in\mathcal{R}^r_{0}(0,\sigma,q)$.
\end{proof}

\section{Diagonalization of the problem}

\subsection{Spaces of diagonal and non diagonal operators}
Owing to
Proposition $\ref{proposition3.3.1}$, the nonlinearity in \eqref{E3.3.4} can be decomposed as the sum of the action of the para-differential potential $\tilde{V}(u,\omega,\epsilon)$ on $u$ and of a remainder.
Thus Eq. \eqref{E3.3.4} can be reduced to
\begin{equation}\label{E3.3.17}
{L}_\omega u+\epsilon {V}(u,\omega,\epsilon)u=\epsilon \tilde{R}(u,\omega,\epsilon)u+\epsilon f,
\end{equation}
where ${L}_\omega:=-\tilde{{L}}_\omega$, $V(u,\omega,\epsilon):=-\tilde{V}(u,\omega,\epsilon)$.
Furthermore, ${V}\in\Sigma^0(0,\sigma,q)$ is symmetric, and $\tilde{R}\in\mathcal{R}^r_0(0,\sigma,q)$. Remark that the symmetric operator ${V}$ is also self-adjoint.

\begin{defi}
Let $\sigma\in\mathbb{R},N\in\mathbb{N}$, with $\sigma\geq \sigma_0+2N+(d+1)/2,\chi\in\mathbb{R},q>0$.
\\ $(i)$  Denote by $\Sigma^\chi_{\rm D}(N,\sigma,q)$ the subspace of $\Sigma^\chi(N,\sigma,q)$ constituted by elements $A(u,\omega,\epsilon)$ satisfying
$\widetilde{{\Pi}}_\alpha A\widetilde{{\Pi}}_{\alpha'}\equiv0$  for any $\alpha,\alpha'\in\mathcal{A}$ with $\alpha\neq \alpha'$.
\\ $(ii)$ Denote by $\Sigma^\chi_{\rm ND}(N,\sigma,q)$ the subspace of $\Sigma^\chi(N,\sigma,q)$ constituted by elements $A(u,\omega,\epsilon)$ satisfying
$\widetilde{{\Pi}}_\alpha A\widetilde{{\Pi}}_\alpha\equiv0$  for any $\alpha\in\mathcal{A}$.
\end{defi}
It is straightforward to see that $\Sigma^\chi(N,\sigma,q)=\Sigma^\chi_{\rm D}(N,\sigma,q)\bigoplus \Sigma^\chi_{\rm ND}(N,\sigma,q)$.
\begin{defi}
Let $\rho'=1+\rho/2>1$.
One denotes by $\mathcal{L}^\chi_{\rho'}(N,\sigma,q)$ the subspace of $\Sigma^\chi(N,\sigma,q)$ constituted by those elements $A(u,\omega,\epsilon)$ with
\begin{equation}\label{E4.1.3}
A(u,\omega,\epsilon)\in\Sigma^{\chi-\rho'}(N,\sigma,q).
\end{equation}
Furthermore, one denotes by $\mathcal{L}'^\chi_{\rho'}(N,\sigma,q)$ the subspace of $\Sigma^\chi(N,\sigma,q)$ constituted by  those elements $A(u,\omega,\epsilon)$ satisfying \eqref{E4.1.3} and
${A(u,\omega,\epsilon)^*}=-A(u,\omega,\epsilon)$.
\end{defi}
\begin{rema}\label{remark4.1.1}
Assume $\sigma\geq\sigma_0+2N+(d+1)/2+\max\{\chi_1+\chi_2-2\rho',0\}$. If $A\in
\mathcal{L}^{\chi_1}_{\rho'}(N,\sigma,q)$, $B\in \mathcal{L}^{\chi_2}_{\rho'}(N,\sigma,q)$, due to Proposition \ref{proposition3.1.1} $(ii)$, then $A\circ B$ is the sum of an element of $\mathcal{L}^{\chi_1+\chi_2-\rho'}_{\rho'}(N,\sigma,q)$, and  an element of $\mathcal{R}^r_0(N,\sigma,q)$ with
$r=\sigma-\sigma_0-2N-(d+1)/2-(\chi_1+\chi_2-2\rho')$.
\end{rema}
\subsection{A class of sequences}

Assume there exists  a class of sequences $S_j(u,\omega,\epsilon),0\leq j\leq N$
satisfying that $S_j$ is written as $S_j=S_{1,j}+S_{2,j}$ with
\begin{equation}\label{E4.2.4}
\begin{split}
&S_{1,j}\in \mathcal{L'}^{-j\rho'}_{\rho'}(j,\sigma,q), \quad [\Delta^2,S_{1,j}]\in\Sigma^{-j\rho'}(j,\sigma,q), \quad j=0,\cdots, N,\\
&S_{2,j}\in \mathcal{L'}^{-(j+1)\rho'}_{\rho'}(j,\sigma,q), \quad [\Delta^2,S_{2,j}]\in\Sigma^{-(j+1)\rho'}(j,\sigma,q), \quad j=0,\cdots, N-1,\\
&S_{2,N}=0.
\end{split}
\end{equation}
Let us check some properties of the class of sequences $S_j(u,\omega,\epsilon),0\leq j\leq N$  satisfying \eqref{E4.2.4}.
\begin{lemm}\label{lemma4.2.1}
Let $r,\sigma,N$ satisfy $(N+1)\rho'\geq r+2$ and $\sigma\geq \sigma_0+2(N+1)+(d+1)/2+r$.
Set \[
S(u,\omega,\epsilon)=\sum\limits^N_{j=0}S_j(u,\omega,\epsilon),\]
 where $S_j=S_{1,j}+S_{2,j}$ and $S_{1,j},S_{2,j}$ satisfy \eqref{E4.2.4}. The following two facts hold:
\\
{(i)}  One may find, for $1\leq j\leq N$, $A_j\in\Sigma^{-j\rho'}(j-1,\sigma,q)$ depending only on $S_l,l\leq j-1$ and satisfying $A_j^*=A_j$, one may find $R\in\mathcal{R}^r_2(N+1,\sigma,q)$, such that
\begin{equation}\label{E4.2.7}
[S^*,{L}_\omega]S+S^*[{L}_\omega,S]=A^N+R,
\end{equation}
where $A^N=\Sigma^N_{j=0}A_j$ with $A_0=0$, $[S^*,{L}_\omega]=S^*{L}_\omega-{L}_\omega S^*$.\\
{(ii)} One may find $A_j,1\leq j\leq N$ as (i), $B_j\in\mathcal{L}^{-(j+1)\rho'}_{\rho'}(j,\sigma,q),0\leq j\leq N-1$, satisfying $[\Delta^2,B_j]\in\Sigma^{-(j+1)\rho'}(j,\sigma,q)$, $B_j$ depending only on $S_{1,l},l\leq j,S_{2,l},l\leq j-1$, and $R\in\mathcal{R}^r_2(N+1,\sigma,q)$, such that
\begin{equation*}\label{E4.2.8}
S^*{L}_\omega S=A^N+{(B^{N-1})^*}{L}_\omega+{L}_\omega B^{N-1}+R,
\end{equation*}
where $B^{N-1}=\Sigma^{N-1}_{j=0}B_j$.
\end{lemm}
\begin{proof}
{(i)} Since $[{L}_\omega,S]=\omega^2[\partial_{tt},S]+[\Delta^2,S]$, the left hand side of \eqref{E4.2.7} equals to
\begin{equation*}
{S^*}[\Delta^2,S]+[S^*,\Delta^2]S+\omega^2{S^*}[\partial_{tt},S]+\omega^2[S^*,\partial_{tt}]S.
\end{equation*}
Let $\widehat{A}:={S^*}[\Delta^2,S]+[S^*,\Delta^2]S$. We
write $\widehat{A}=\Sigma^{2N+1}_{j=1}\widehat{A}_j$, where
\begin{equation}\label{E4.2.10}
\widehat{A}_j:=\sum\limits_{\stackrel{j_1+j_2=j-1}{0\leq j_1,j_2\leq
N}}\left([S^*_{j_1},\Delta^2]S_{j_2}+{S^*_{j_2}}[\Delta^2,S_{j_1}]\right).
\end{equation}
It follows from  \eqref{E4.2.4} and Proposition \ref{proposition3.1.1} (ii) that $\widehat{A}_j$ may be written as the sum $A_j+R_j$, where
\begin{equation*}
A_j\in\Sigma^{-j\rho'}(\min\{N,j-1\},\sigma,q), \quad ~R_j\in\mathcal{R}^{r_1}_0(\min\{N,j-1\},\sigma,q)
\end{equation*}
with
$r_1=\sigma-\sigma_0-2N-(d+1)/2+j\rho'\geq r$.
Formula \eqref{E4.2.10} implies that $A_j$ depends only on $S_l,l\leq j-1$  and that $A_j$ is self-adjoint. Furthermore,
$A_j$ is in $\Sigma^{-(N+1)\rho'}(N,\sigma,q)$ for $j\geq N+1$, hence in $\mathcal{R}^r_0(N,\sigma,q)$ thanks to the inequality $(N+1)\rho'\geq r$
and Remark \ref{remark3.1.2}. On the other hand, denote $\widehat{B}:=\omega^2{S^*}[\partial_{tt},S]+\omega^2[S^*,\partial_{tt}]$.
Clearly, $\widehat{B}$ is written as $\Sigma^{2N+2}_{j=2}\widehat{B}_j$, where
\begin{equation}\label{E4.2.12}
\widehat{B}_j=\omega^2\sum\limits_{\stackrel{j_1+j_2=j-2}{0\leq j_1,j_2\leq N}}\big({S^*_{j_1}}[\partial_{tt},S_{j_2}]+[S^*_{j_2},\partial_{tt}]S_{j_1}\big).
\end{equation}
 According to formula \eqref{E4.2.4}, Remark \ref{remark3.1.1} and Proposition \ref{proposition3.1.1} {\rm(ii)}, it yields that  $\widehat{B}_j$ may be written as the sum
$A_j+R_j$, where
\begin{equation*}
A_j\in\Sigma^{-j\rho'}(\min\{N+1,j-1\},\sigma,q), R_j\in\mathcal{R}^{r_2}_0(\min\{N+1,j-1\},\sigma,q)
\end{equation*}
with
$r_2=\sigma-\sigma_0-2(N+1)-(d+1)/2+j\rho'\geq r$.
By formula \eqref{E4.2.12}, we have that $A_j$
depends only on $S_l,l\leq j-2$ and that $A_j$ is self-adjoint. In addition, $A_j\in\Sigma^{-(N+1)\rho'}(N+1,\sigma,q)$ for $j\geq N+1$, hence in
$\mathcal{R}^r_0(N+1,\sigma,q)$. Set $A^N=\Sigma^N_{j=0}A_j$ with $A_0=0$. This concludes the proof.
\\\\
{(ii)} We express ${S^*}{L}_\omega S$ in terms of the sum of the following
\begin{equation*}\label{E4.2.13}
{\textstyle\frac{1}{2}}\left(S^*[{L}_\omega,S]+[S^*,{L}_\omega]S\right)+
{\textstyle\frac{1}{2}}\left(S^*S{L}_\omega+{L}_\omega {S^*}S\right).
\end{equation*}
From the proof of {(i)}, the term ${\frac{1}{2}}\left(S^*[{L}_\omega,S]+[S^*,{L}_\omega]S\right)$ is written as $A^N+R$. We write ${S^*}S$ as the sum in $j$ of
\begin{equation}\label{E4.2.14}
\sum\limits_{\stackrel{j_1+j_2=j}{0\leq j_1,j_2\leq N}}{S_{1,j_1}}^*S_{1,j_2}+\sum\limits_{\stackrel{j_1+j_2=j-1}{0\leq j_1,j_2\leq N}}\big({{S_{1,j_1}}^*}S_{2,j_2}
+{{S_{2,j_1}}^*}S_{1,j_2}\big)+\sum\limits_{\stackrel{j_1+j_2=j-2}{0\leq j_1,j_2\leq N}}{{S_{2,j_1}}^*}S_{2,j_2}.
\end{equation}
Formula \eqref{E4.2.4} and Remark \ref{remark4.1.1} shows that the term in \eqref{E4.2.14} may be written as $B_j+R_j$, where
\begin{equation*}
B_j\in\mathcal{L}^{-(j+1)\rho'}_{\rho'}(\min\{N,j\},\sigma,q), \quad ~R_j\in\mathcal{R}^{r_3}_0(\min\{N,j\},\sigma,q)
\end{equation*}
 with
$r_3=\sigma-\sigma_0-2N-(d+1)/2+(j+2)\rho'\geq r+2$. Furthermore, we have $B_j\in\Sigma^{-(N+1)\rho'}(N,\sigma,q)$  for $j\geq N$. It follows from the inequality $(N+1)\rho'\geq r+2$ and Remark \ref{remark3.1.2} that $B_j$ belongs to $\mathcal{R}^r_0(N,\sigma,q)$. The expression in
\eqref{E4.2.14} indicates that $B_j$ depends only on $S_{1,l}, l\leq j$, $S_{2,l}, l\leq j-1$. Formula \eqref{E4.2.4} shows that
\begin{equation*}
[\Delta^2,B_j]\in\Sigma^{-(j+1)\rho'}(\min\{N,j\},\sigma,q).
\end{equation*}
Remark that $B_j{L}_\omega,{L}_\omega B_j$ for $j\geq N+1$ are in $\mathcal{R}^r_2(N,\sigma,q)$ and that $R_j{L}_\omega,{L}_\omega R_j$ for $j\geq 0$ belong to $\mathcal{R}^r_2(N,\sigma,q)$. At last, we set $B^{N-1}=\Sigma^{N-1}_{j=0}B_j$.
\end{proof}
\begin{prop}\label{proposition4.2.2}
Let $r,\sigma,N,S(u,\omega,\epsilon)$ satisfy the conditions of Lemma \ref{lemma4.2.1}. \\
(i) There are elements
\begin{equation*}\label{E4.2.5}
B_j(u,\omega,\epsilon)\in\mathcal{L}^{-(j+1)\rho'}_{\rho'}(j,\sigma,q),0\leq j\leq N-1~\text{with}~
[\Delta^2,B_j]\in\Sigma^{-(j+1)\rho'}(j,\sigma,q),
\end{equation*}
where $B_j$ depends only on $S_{1,l},l\leq j,S_{2,l},l\leq j-1$;\\
(ii) There are elements
\begin{equation*}
V_{j}(u,\omega,\epsilon)\in\Sigma^{-j\rho'}(j,\sigma,q),0\leq j\leq N ~\text{with}~{V_{j}}^*=V_{j}
\end{equation*}
where  $V_{j}$ depends only on $S_l,l\leq j-1$;\\
(iii) There is an element $R\in\mathcal{R}^r_2(N+1,\sigma,q)$, such that if denote
\begin{equation*}
V^N(u,\omega,\epsilon)=\sum\limits_{j=0}^N V_{j}(u,\omega,\epsilon), \quad ~B^{N-1}(u,\omega,\epsilon)=\sum\limits_{j=0}^{N-1}B_j(u,\omega,\epsilon),
\quad ~S_i=\sum\limits_{j=0}^{N}S_{i,j},i=1,2,
\end{equation*}
then we have
\begin{align}
\big({\rm Id}+\epsilon S\big)^*\big(L_\omega+\epsilon V\big)
\big({\rm Id}+\epsilon S\big)
=&L_\omega+\epsilon V^N+\epsilon\big({(B^{N-1})^*}L_\omega+L_\omega(B^{N-1})\big) \nonumber\\
&+\epsilon\big(S_1^*(\Delta^2+m)+(\Delta^2+m)S_1\big)
+\epsilon\big(S_2^*L_\omega+L_\omega S_2\big)+\epsilon R.\label{E4.2.6}
\end{align}

\end{prop}
\begin{proof}
The left hand side of \eqref{E4.2.6} may be expressed in terms of the sum of the following
\begin{subequations}
\begin{numcases}{}
\label{E4.2.15a}
L_\omega+\epsilon V(u,\omega,\epsilon)+\epsilon^2{S^*}L_\omega S,\\
\label{E4.2.15b}
\epsilon({S_1^*}(\omega^2\partial_{tt})+(\omega^2\partial_{tt})S_1),\\
\label{E4.2.15c}
\epsilon^2({S^*}V+VS)+\epsilon^3{S^*}VS,\\
\label{E4.2.15d}
\epsilon({S_1^*}(\Delta^2+m)+(\Delta^2+m)S_1)
+\epsilon({S_2^*}L_\omega+L_\omega S_2).
\end{numcases}
\end{subequations}
\\
In \eqref{E4.2.15a}, the term $V$ contributes to the $V_0$ component of $V^N$. Lemma \ref{lemma4.2.1}
shows that the $A_j$ component of $A^N$ contributes to the $V_j$ component of $V^N$ and that the $B_j$
satisfies the condition of Proposition \ref{proposition4.2.2}. We write the term in \eqref{E4.2.15b} as the sum in $j$ of
${{S_{1,j-1}}^*}(\partial_{tt})+(\partial_{tt}){{S_{1,j-1}}}$,
which is self-adjoint. Remark \ref{remark3.1.1} infers that
\begin{equation*}
\omega^2({{S_{1,j-1}}^*}(\partial_{tt})+(\partial_{tt}){{S_{1,j-1}}})\in\Sigma^{-j\rho'}(j,\sigma,q).
 \end{equation*}
Then we get a contribution to $V_j$ for $1\leq j\leq N$.  Combining this with Lemma \ref{lemma4.2.1}, it is obvious to read that $V_j$ depends only on $S_{1,l},l\leq j-1,S_{2,l},l\leq j-2$. Owing to the inequality $(N+1)\rho'\geq r+2$ and Remark \ref{remark3.1.2}, it yields that
\begin{equation*}
\omega^2({S_{1,N}^*}(\partial_{tt})+(\partial_{tt}){{S_{1,N}}})\in\mathcal{R}^r_0(N+1,\sigma,q).
\end{equation*}
The term \eqref{E4.2.15c} may be expressed as the sum in $j$ of
\begin{align}\label{E4.2.17}
{{S_{1,j-1}}^*}V+VS_{1,j-1}&+{{S_{2,j-2}}^*}V+VS_{2,j-2}+\epsilon\sum\limits_{j_1+j_2=j-2}{{S_{1,j_1}}^*}VS_{1,j_2} \nonumber\\
&+\epsilon\sum\limits_{j_1+j_2=j-3}({{S_{2,j_1}}^*}VS_{1,j_2}+{{S_{1,j_1}}^*}VS_{2,j_2}) \nonumber\\
&+\epsilon\sum\limits_{j_1+j_2=j-4}{{S_{2,j_1}}^*}VS_{2,j_2}.
\end{align}
Applying the facts of $S_{1,j}\in\Sigma^{-(j+1)\rho'}(j,\sigma,q)$, $S_{2,j}\in\Sigma^{-(j+2)\rho'}(j,\sigma,q)$ and
$V\in\Sigma^0(0,\sigma,q)$, the term in \eqref{E4.2.17} is written as $V_j+R_j$, where
\begin{equation*}
V_j\in\Sigma^{-j\rho'}(\min\{N,j-1\},\sigma,q),
R_j\in \mathcal{R}^r_0(\min\{N,j-1\},\sigma,q).
 \end{equation*}
Moreover, we derive that $V_j\in \mathcal{R}^r_0(N+1,\sigma,q)$ for $j\geq N+1$ using the inequality $(N+1)\rho'\geq r+2$ and Remark \ref{remark3.1.2}.
\end{proof}
In addition, we have to give an extra proposition on an self-adjoint element of $\Sigma^\chi_{\rm ND}(N,\sigma,q)$.
\begin{prop}\label{proposition4.1.1}
Assume $\sigma\geq\sigma_0+2N+\frac{d+1}{2}+\frac{1}{\rho'}\max{\{\chi,0\}}$. Let $A(u,\omega,\epsilon)\in\Sigma^\chi_{\rm ND}(N,\sigma,q)$ be self-adjoint. There is an element $B(u,\omega,\epsilon)$ of $\mathcal{L}'^\chi_{\rho'}(N,\sigma,q)$ and an element $R(u,\omega,\epsilon)$ of $\mathcal{R}^{r'-\chi}_0(N,\sigma,q)$, where $r'=\rho'(\sigma-\sigma_0-2N-d-1)$,
such that
\begin{equation}\label{E4.1.5}
B(u,\omega,\epsilon)^*(\Delta^2+m)+(\Delta^2+m)B(u,\omega,\epsilon)
=A(u,\omega,\epsilon)+R(u,\omega,\epsilon).
\end{equation}
Moreover, $[\Delta^2,B]\in\Sigma^\chi(N,\sigma,q)$.
\end{prop}
\begin{proof}
Let $A_1(u,\omega,\epsilon)\in\Sigma^\chi_{\rm ND}(N,\sigma,q)$ with $A_1^*=A_1$. Define
\begin{align*}
{A}=\sum\limits_{n,n'\in\mathbb{Z}^d}\mathds{1}_{|n-n'|\leq c_0(|n|+|n'|)^{\rho}}\Pi_n {A}_1\Pi_{n'},~R=\sum\limits_{n,n'\in\mathbb{Z}^d}\mathds{1}_{|n-n'|> c_0(|n|+|n'|)^{\rho}}\Pi_n {A}_1\Pi_{n'},
\end{align*}
where $c_0$ is small enough. Applying \eqref{E3.1.1} with $M=\sigma-\sigma_0-2N$ and the inequality $|n-n'|> c_0(|n|+|n'|)^{\rho'}$, we obtain that, for all $(\eta_1,\eta_2)\in\mathbb{N}^2$ with $\eta_1+\eta_2\leq1$,
\begin{align*}
\left\|\Pi_n\partial^{\eta_1}_\epsilon\partial^{\eta_2}_\omega\mathrm{D}^j_u {A}''_2(u)\cdot(w_1,...,w_j)\Pi_{n'}\right\|_{\mathcal{L}(\mathcal{H}^0)}\leq
&C(1+|n|+|n'|)^{-{2}(r'-\chi)} \langle n-n'\rangle^{-(d+1)}\\
&\times\mathds{1}_{|n-n'|\leq \frac{1}{10}(|n|+|n'|)}
\prod\limits_{l=1}^j \|w_l\|_{\mathcal{H}^\sigma},
\end{align*}
where $r'=\rho'(\sigma-\sigma_0-2N-(d+1)/2)$. The same argument as the proof of lemma \ref{lemma3.1.1} implies that $R\in\mathcal{R}^{r'-\chi}_0(N,\sigma,q)$. Evidently, it can be seen that $A$ is in $\Sigma^\chi_{\rm ND}(N,\sigma,q)$ with $A=A^*$. Remark that $A=0$ if $|n-n'|> c_0(|n|+|n'|$. Assume there exists $B(u,\omega,\epsilon)\in\mathcal{L}^\chi_{\rho'}(N,\sigma,q)$ with $B^*=-B$  such that formula \eqref{E4.1.5} holds. Equivalently,  we have to solve the equation $[\Delta^2,B]=A$, which then indicates that $[\Delta^2,B]\in\Sigma^\chi(N,\sigma,q)$. This is also equivalent to $(|n|^4-|n'|^4)\Pi_n B\Pi_{n'}=\Pi_n A\Pi_{n'}$.
Using the definition of $\Sigma^\chi_{\rm ND}(N,\sigma,q)$, we define
\begin{align*}
B(u,\omega,\epsilon)&=\sum\limits_{\alpha,\alpha'\in\mathcal{A},\alpha\neq\alpha'}
\sum\limits_{n\in\Omega_\alpha}\sum\limits_{n'\in\Omega_{\alpha'}}
(|n|^4-|n'|^4)^{-1}\Pi_n A(u,\omega,\epsilon)\Pi_{n'}.
\end{align*}
Owing to formula \eqref{E2.4} and $|n-n'|> c_0(|n|+|n'|)^{\rho}$, for $n\in\Omega_\alpha,n'\in\Omega_{\alpha'}$ with $\alpha\neq\alpha'$, we have
\begin{equation*}
||n^2|-|n'|^2|\geq c(|n|+|n'|)^{\rho}
\end{equation*}
for some constant $c>0$. Combining this with \eqref{e1.3}, we obtain that
\begin{equation*}
\left|\left(\sqrt{n^4+m}-\sqrt{{n'}^4+m}\right)\left(\sqrt{n^4+m}+\sqrt{{n'}^4+m}\right)\right|\geq  c(m,\rho)(1+|n|+|n'|)^{2+\rho},
\end{equation*}
 Thus we have  $B\in\Sigma^{\chi-\rho'}(N,\sigma,q)$, where $\rho'=1+\rho/2>1$.
\end{proof}

\subsection{Diagonalization theorem}
The following proposition gives a reduction for operator ${L}_\omega+\epsilon V$ in \eqref{E3.3.17}. Through the para-differential conjugation, the para-differential potential $V(u,\omega,\epsilon)$ is replaced by $V_{\rm D}(u,\omega,\epsilon)$, where $V_{\rm D}$ is block diagonal relatively to an orthogonal decomposition of $L^2(\mathbb{T}^d)$ in a sum of finite dimensional subspaces.
\begin{prop}
Let $q>0$, $r>0$ and $N\in\mathbb{N}$ with $(N+1)\rho'\geq r+2$, $\sigma\in\mathbb{R}$ with
$\sigma\geq \sigma_0+2(N+1)+(d+1)/2+{r}$.
One can find elements $Q_j(u,\omega,\epsilon) \in\mathcal{L}_{\rho'}^{-j\rho'}(j,\sigma,q),0\leq j \leq N$, elements $V_{{\rm D},j}(u,\omega,\epsilon) \in\Sigma^{-j\rho'}_{\rm D}(j,\sigma,q),0\leq j \leq N$, an element $R_1(u,\omega,\epsilon)\in\mathcal{R}^r_2(N+1,\sigma,q)$,  where $Q_j(u,\omega,\epsilon), V_{{\rm D},j}(u,\omega,\epsilon),R_1(u,\omega,\epsilon)$ are also $C^1$ in $(\omega,\epsilon)$, such that for any $u\in B_q(\mathcal{H}^\sigma)$, this holds:
\begin{equation}\label{E4.2.3}
({\rm Id}+\epsilon Q(u,\omega,\epsilon))^*({L}_\omega+
\epsilon V(u,\omega,\epsilon))({\rm Id}+\epsilon Q(u,\omega,\epsilon))
={L}_\omega+\epsilon V_{\rm D}(u,\omega,\epsilon)-\epsilon R_1(u,\omega,\epsilon),
\end{equation}
 where
\begin{equation}\label{E4.2.2}
\begin{array}{ll}
Q(u,\omega,\epsilon)=\sum\limits_{j=0}^N Q_j(u,\omega,\epsilon),\quad ~V_{\rm D}(u,\omega,\epsilon)=\sum\limits_{j=0}^{N} V_{{\rm D},j}(u,\omega,\epsilon).
\end{array}
\end{equation}
\end{prop}
\begin{proof}
Let us verify that the right hand side of \eqref{E4.2.6} may be written as the right hand side of \eqref{E4.2.3}.
Assume that $Q_0,\cdots,Q_{j-1}$ , where $Q_{i},0\leq i\leq j-1$
may be written as the sum $Q_{1,i}+Q_{2,i}$ with $Q_{1,i},Q_{2,i}$ satisfying \eqref{E4.2.4},
such that $V_{j}$ may be determined ($V_{j}$ depends only on $Q_l,l\leq j-1$) and the right hand side of \eqref{E4.2.6} can be written as
\begin{align}\label{E4.2.18}
{L}_\omega&+\epsilon\sum\limits_{j'=0}^{j-1}V_{{\rm D},j'}+\epsilon\sum\limits_{j'=j}^{N-1}(B^*_{j'}{L}_\omega+{L}_\omega B_{j'}) \nonumber\\
&+\epsilon\sum\limits_{j'=j}^{N}({{Q_{1,j'}}^*}(\Delta^2+m)+(\Delta^2+m)Q_{1,j'}) \nonumber\\
&+\epsilon\sum\limits_{j'=j}^{N-1}({{Q_{2,j'}}^*}{L}_\omega+{L}_\omega Q_{2,j'}) \nonumber\\
&+\epsilon\sum\limits_{j'=j}^{N}V_{j'}+\epsilon R.
\end{align}
It is straightforward to show that \eqref{E4.2.18} with $j=0$ is the conclusion of Proposition \ref{proposition4.2.2}. Since $V_{j}\in\Sigma^{-j\rho'}(j,\sigma,q)$ with $V^*_{j}=V_{j}$, depending on $Q_l,l\leq{j-1}$, we define
\begin{align*}
V_{{\rm D},j}=\sum\limits_{\alpha\in\mathcal{A}}\widetilde{{\Pi}}_\alpha V_{j}\widetilde{{\Pi}}_{\alpha'},\quad
V_{{\rm ND},j}=\sum\limits_{{\alpha,\alpha'\in\mathcal{A}},{\alpha\neq \alpha'}}\widetilde{{\Pi}}_\alpha V_{j}\widetilde{{\Pi}}_{\alpha'}.
\end{align*}
Then $V_{{\rm D},j}\in\Sigma^{-j\rho'}_{\rm D}(j,\sigma,q)$ with $(V_{{\rm D},j})^*=V_{{\rm D},j}$ and $V_{{\rm ND},j}\in\Sigma^{-j\rho'}_{\rm ND}(j,\sigma,q)$ with $(V_{{\rm ND},j})^*=V_{{\rm ND},j}$, where $V_{{\rm ND},j}$ depends only on $Q_l,l\leq{j-1}$. It follows from Proposition \ref{proposition4.1.1} that, for $V_{{\rm ND},j}$, we may find $C_{j}\in\mathcal{L}'^{-j\rho'}_{\rho'}(j,\sigma,q)$ such that ${C_{j}}^*(\Delta^2+m)+(\Delta^2+m){C_{j}}=V_{{\rm ND},j}+R$ with $[\Delta^2,C_{j}] \in\Sigma^{-j\rho'}(j,\sigma,q)$. Let $Q_{1,j}:=-C_{j}$. This shows that we may eliminate the $j$th component of
\begin{align*}
\epsilon\sum\limits_{j'=j}^{N}({{Q_{1,j'}}^*}(\Delta^2+m)+(\Delta^2+m)Q_{1,j'})
\end{align*}
and $\epsilon\sum_{j'=j}^{N}V_{j'}$. Moreover, $Q_{1,j}$ satisfies \eqref{E4.2.4}.
Set $Q_{2,j}:=-B_j,0\leq j\leq N-1$. Then we may eliminate the $j$th component of
$\epsilon\sum_{j'=j+1}^{N}(B^*_{j'}
{L}_\omega+{L}_\omega B_{j'})$
and
\begin{align*}
\epsilon\sum_{j'=j+1}^{N}({{Q_{2,j'}}^*}{L}_\omega+{L}_\omega Q_{2,j'}).
\end{align*}
 In addition, $Q_{2,j}$ satisfies (\ref{E4.2.4}).
Therfore we may construct recursively $Q_{1,j},0\leq j\leq N$, $Q_{2,j},0\leq j\leq N-1$ satisfying \eqref{E4.2.4}, such that the equality in \eqref{E4.2.3} holds.
\end{proof}

\section{Iterative scheme} \label{sec:6}

This section  concerns with the proof of Theorem \ref{theorem2.1.1}. Firstly, we investigate  some properties about the restriction of
the operator $L_\omega+\epsilon V_{\rm D}(u,\omega,\epsilon)$ to ${\rm Range}(\widetilde{{\Pi}}_\alpha)$.
Next, under the non-resonant conditions \eqref{E5.1.6}, we prove the restriction is invertible and the frequencies $\omega$ are in a Cantor-like set whose complement has small measure. Finally, we use a standard iterative scheme to construct the solutions.		
\subsection{Lower bounds for eigenvalues}
Let $\gamma_0\in(0,1], \sigma\in\mathbb{R}, N\in\mathbb{N},\zeta\in\mathbb{R}_+ $ with $\sigma \geq \sigma_0+2(N+1)+(d+1)/2+{\zeta}$. Define the space of functions by
\begin{align*}
\mathcal{E}^\sigma_\zeta:=
&\mathcal{E}^\sigma_\zeta(\mathbb{T}\times \mathbb{T}^d\times [1,2]\times(0,\gamma_0];\mathbb{R})
=\bigg\{u(t,x,\omega,\epsilon);~u\in \mathcal{H}^\sigma,~\partial_\omega u\in{\mathcal{H}^{\sigma-\zeta-2}},\\
&~ u,~
\partial_\omega u~ \text{ ~are~continuous~in~}
\omega~(\text{uniformly}~\text{for}~ \epsilon\in[0,\gamma_0]),~\|u\|_{\mathcal{E}^\sigma_\zeta}<+\infty
\bigg\}.
\end{align*}
where
\[\|u\|_{\mathcal{E}^\sigma_\zeta}
:=\sup\limits_{(\omega,\epsilon)\in [1,2]\times[0,\gamma_0]}\|u(\cdot,\omega,\epsilon)\|_{\mathcal{H}^\sigma}+
\sup\limits_{(\omega,\epsilon)\in[1,2]\times[0,\gamma_0]}\|\partial_\omega u(\cdot,\omega,\epsilon)\|_{\mathcal{H}^{\sigma-\zeta-2}}.
\]
Moreover, for fixed $\alpha\in\mathcal{A}$,  we define a self-adjoint operator as
\begin{equation}\label{E5.1.3}
A_\alpha(\omega;u,\epsilon)=\widetilde{{\Pi}}_\alpha({L}_\omega+\epsilon V_{\rm D}(u,\omega,\epsilon))\widetilde{{\Pi}}_\alpha
\end{equation}
for all $u\in\mathcal{E}^\sigma_\zeta,\omega\in[1,2],\epsilon\in(0,\gamma_0]$. Denote $F_\alpha={\rm Range}(\widetilde{{\Pi}}_\alpha)$, $D_\alpha={\rm dim} (F_\alpha)$. Formula \eqref{E2.2.8} derives that $D_\alpha \leq C\langle n(\alpha)\rangle^{\beta d+2}$ for some $C>0$. This implies that $A_\alpha(\omega;u,\epsilon)$ is defined on a space of finite dimension. By means of \eqref{E4.2.2}, \eqref{E3.1.1}, $\partial_\omega u\in\mathcal{H}^{\sigma-\zeta-2}$ and the fact of $\sigma \geq \sigma_0+2(N+1)+(d+1)/2+{\zeta}$ , it yields  that $\widetilde{{\Pi}}_\alpha(V_{\rm D}(u,\omega,\epsilon))\widetilde{{\Pi}}_\alpha$ is $C^1$ in $\omega$, which then gives that $A_\alpha(\omega;u,\epsilon)$ is $C^1$ in $\omega$.

\begin{prop}\label{proposition5.1.1}
Let $m>0$, $q>0$. There exists $\gamma_0\in(0,1]$ small enough, $C_0>0$, such that, for all $ \epsilon\in[0,\gamma_0]$, $u\in\mathcal{E}^\sigma(\zeta)$ with $\|u\|_{\mathcal{E}^\sigma (\zeta)}<q$, $\alpha\in\mathcal{A}$, the eigenvalues of $A_\alpha$ form a finite family of $C^1$ real valued functions of $\omega$, depending on $(u,\epsilon)$, i.e.
\begin{equation*}\label{E5.1.4}
\omega\rightarrow \lambda^\alpha_l(\omega;u,\epsilon), \quad 1\leq l\leq D_\alpha
\end{equation*}
and satisfy\\
$(i)$   For all $\alpha\in\mathcal{A}$, $l\in\{1,...,D_n\}$, $\epsilon\in(0,\gamma_0]$, $\omega\in[1,2]$, $u,u'\in\mathcal{H}^\sigma$ with $\|u\|_{\mathcal{H}^{\sigma}},\|u'\|_{\mathcal{H}^{\sigma}}$ smaller than $q$, we have $l'\in\{1,...,D_\alpha\}$ satisfying
\begin{equation}\label{E5.1.5}
|\lambda^\alpha_l(\omega;u,\epsilon)-\lambda^\alpha_{l'}(\omega;u',\epsilon)|\leq C_0\epsilon\|u-u'\|_{\mathcal{H}^{\sigma}};
\end{equation}
$(ii)$  For all $\alpha\in\mathcal{A}$, $ \epsilon\in(0,\gamma_0]$, $l\in\{1,...,D_\alpha\}$, any $\omega\in[1,2]$, $u\in\mathcal{E}^\sigma(\zeta)$ with $\|u\|_{\mathcal{E}^\sigma(\zeta)}<q$, this holds:
\begin{equation}\label{E5.1.6}
-4C_0\langle n(\alpha)\rangle^4\leq{\partial_\omega \lambda^\alpha_l}(\omega;u,\epsilon)\leq -2C^{-1}_0\langle n(\alpha)\rangle^4;
\end{equation}
$(iii)$ For all $\alpha\in\mathcal{A}$, $\delta\in(0,1]$, $\epsilon\in(0,\gamma_0]$, $u\in\mathcal{E}^\sigma_\zeta$ with $\|u\|_{\mathcal{E}^\sigma_\zeta}<q$, if we set
\begin{equation}\label{E5.1.7}
I(\alpha,u,\epsilon,\delta)=\left\{\omega\in[1,2];~\forall~ l\in\{1,...,D_\alpha\},~|\lambda^\alpha_l(\omega;u,\epsilon)|\geq\delta\langle n(\alpha)\rangle^{-2\zeta}\right\},
\end{equation}
 then there exists a constant $E_0$ depending only on the dimension, such that, for all $\omega\in I(n,u,\epsilon,\delta)$, the operator $A_\alpha(\omega;u,\epsilon)$ is invertible and satisfies
\begin{equation}\label{E5.1.8}
\|A_\alpha(\omega;u,\epsilon)^{-1}\|_{\mathcal{L}(\mathcal{H}^0)}\leq E_0\delta^{-1}\langle n(\alpha)\rangle^{2\zeta},~
\|\partial_\omega (A_\alpha(\omega;u,\epsilon)^{-1})\|_{\mathcal{L}(\mathcal{H}^0)}\leq E_0\delta^{-2}\langle n(\alpha)\rangle^{4\zeta+4}.
\end{equation}
\end{prop}
\begin{proof}
{\rm(i)} According to that $A_\alpha(\omega;u,\epsilon)$ is defined on a space of finite dimension, Theorem $6.8$ in {\cite{kato1995perturbation}}
shows that we may index eigenvalues $\lambda^\alpha_l(\omega;u,\epsilon),l\in\{1,...,D_n\}$
of $A_\alpha$ such that they are $C^1$ functions of $\omega$. Moreover, if $B,B'$ are self-adjoint operators in the same dimension space, then for any eigenvalue $\lambda_l(B)$ of $B$, there is an eigenvalue $\lambda_{l'}(B')$ of $B'$ with $l'\in\{1,...,D_n\}$ such
that $|\lambda_l(B)-\lambda_{l'}(B')|\leq\|B-B'\|$. By means of \eqref{E4.2.2}, \eqref{E3.1.1}, we obtain that $u\rightarrow A_\alpha(\omega;u,\epsilon)$ is lipschitz with values in $\mathcal{L}(\mathcal{H}^0)$. Consequently, formula \eqref{E5.1.5} can be obtained with lipschitz constant $C_0\epsilon$ using \eqref{E5.1.3}.\\
{\rm(ii)} Denote $L^\alpha_\omega=\widetilde{\Pi}_\alpha L_\omega\widetilde{\Pi}_\alpha$. For $\alpha\in\mathcal{A}$, let $\Lambda(\alpha),\Lambda^0(\alpha)$ stand for the spectrum set of $A_\alpha, L^\alpha_\omega$ respectively, where
\begin{equation*}
\Lambda^0{(\alpha)}=\left\{-\omega^2j^2+{|n|}^4+m: ~ n\in\Omega_\alpha, ~j\in\mathbb{Z}~\text{ with}~K^{-1}_0\langle n(\alpha)\rangle^2\leq |j|\leq K_0\langle n(\alpha)\rangle^2\right\}.
\end{equation*}
Let $\Gamma$ be a contour in the complex plane turning once around $\Lambda^0(\alpha)$,
of length $O(\langle n(\alpha)\rangle^4)$, where $\Gamma$ satisfies ${\rm dist}(\Gamma,\Lambda^0(\alpha))\geq c\langle n(\alpha)\rangle^4$.
If $\epsilon\in[0,\gamma_0]$ with $\gamma_0$ small enough, then we also have ${\rm dist}(\Gamma,\Lambda(\alpha))\geq c\langle n(\alpha)\rangle^4$.
Moreover, we define the spectral projectors $\Pi_\alpha(\omega)$, $\Pi^0_\alpha(\omega)$, which are
associated to the eigenvalues of $A_\alpha$, $L^\alpha_\omega$ respectively,  by
\begin{equation}\label{E5.1.9}
\Pi_\alpha(\omega)=\frac{1}{2\mathrm{i}\pi}\int_\Gamma(\zeta {\rm Id}-A_\alpha)^{-1}{\rm d}\zeta,\quad
\Pi^0_\alpha=\frac{1}{2\mathrm{i}\pi}\int_\Gamma(\zeta {\rm Id}-L^\alpha_\omega)^{-1} {\rm d}\zeta.
\end{equation}
Then there exist some constant $C>0$ such that
\begin{equation}\label{E5.1.10}
\|\Pi_\alpha(\omega)\|_{\mathcal{L}(F_\alpha)}\leq C,\quad \|\Pi^0_\alpha\|_{\mathcal{L}(F_\alpha)}\leq C.
\end{equation}
Remark that $\Pi^0_\alpha$ is just the orthogonal projector on
\begin{equation*}
{\rm Vect}\left\{e^{{\rm i}(jt+n\cdot x)}: ~n\in\Omega_{\alpha}, ~j\in\mathbb{Z}~\text{ with}~ K^{-1}_0\langle n(\alpha)\rangle^2\leq |j|\leq K_0\langle n(\alpha)\rangle^2\right\}.
\end{equation*}
 This indicates that $\Pi^0_\alpha$ is independent of $\omega$. Let us estimate the upper bound of
$\|\partial_\omega(\Pi_\alpha(\omega)A_\alpha\Pi_\alpha(\omega)-\Pi^0_\alpha L^\alpha_\omega\Pi^0_\alpha)\|_{\mathcal{L}(F_\alpha)}$,
 where
\begin{align*}
\Pi_\alpha(\omega)A_\alpha\Pi_\alpha(\omega)-\Pi^0_\alpha L^\alpha_\omega\Pi^0_\alpha=&
(\Pi_\alpha(\omega)-\Pi^0_\alpha)A_\alpha\Pi_\alpha(\omega)+\Pi^0_\alpha(A_\alpha-L^\alpha_\omega)\Pi_\alpha(\omega)\\
&+\Pi^0_\alpha L^\alpha_\omega(\Pi_\alpha(\omega)-\Pi^0_\alpha).
\end{align*}
It follows from \eqref{E5.1.3} and the definition of $L^\alpha_\omega$ that
\begin{equation}\label{E5.1.11}
\|A_\alpha-L^\alpha_\omega\|_{\mathcal{L}(F_\alpha)}+\|\partial_\omega (A_\alpha-L^\alpha_\omega)\|_{\mathcal{L}(F_\alpha)}{\leq} C\epsilon,\quad
\|\partial_\omega A_\alpha\|_{\mathcal{L}(F_\alpha)}+\|\partial_\omega L^\alpha_\omega\|_{\mathcal{L}(F_\alpha)}{\leq} C\langle n(\alpha)\rangle^4.
\end{equation}
 Formula \eqref{E5.1.9} indicates
\begin{equation*}
\Pi_\alpha(\omega)-\Pi^0_\alpha=\frac{1}{2{\rm i}\pi}\int_\Gamma(\zeta {\rm Id}-A_\alpha)^{-1}
(A_\alpha-L^\alpha_\omega)(\zeta {\rm Id}-L^\alpha_\omega)^{-1}~{\rm d}\zeta.
\end{equation*}
Combining this with formula \eqref{E5.1.11}, we give that
\begin{align}\label{E5.1.12}
\|\Pi_\alpha(\omega)-\Pi^0_\alpha\|_{\mathcal{L}(F_\alpha)}{\leq} C\epsilon\langle n(\alpha)\rangle^{-4},\quad ~\|\partial_\omega \Pi_\alpha(\omega)\|_{\mathcal{L}(F_\alpha)}
\leq C\epsilon\langle n(\alpha)\rangle^{-4}.
\end{align}
 Consequently, thanks to \eqref{E5.1.10}, \eqref{E5.1.11}, \eqref{E5.1.12} and the facts of $\|A_\alpha\|_{\mathcal{L}(F_\alpha)}\leq C\langle n(\alpha)\rangle^{4}$ and $\|L^\alpha_\omega\|_{\mathcal{L}(F_\alpha)}\leq C\langle n(\alpha)\rangle^{4}$, it can be seen that
 \begin{equation}\label{E5.1.14}
 \|\partial_\omega(\Pi_\alpha(\omega)A_\alpha\Pi_\alpha(\omega)-\Pi^0_\alpha L^\alpha_\omega\Pi^0_\alpha)\|_{\mathcal{L}(F_\alpha)}\leq C\epsilon.
 \end{equation}
 Let $\mathcal{C}$ be a subinterval of $[1, 2]$. For $\omega\in\mathcal{C}$,
 one of the eigenvalues $\lambda^\alpha_l(\omega;u,\epsilon)$ of $\Pi_\alpha(\omega)A_\alpha\Pi_\alpha(\omega)$ has constant multiplicity $\varsigma$. In addition, $P(\omega)$ stands for the associated spectral projector, where $P(\omega)^2=P(\omega)$ with $C^1$ dependence in $\omega\in \mathcal{C}$. Then \begin{equation*}
\lambda^\alpha_l(\omega;u,\epsilon)={1}/{\varsigma}\left[ {\rm tr}(P(\omega)\Pi_\alpha(\omega)A_\alpha\Pi_\alpha(\omega)P(\omega))\right].
\end{equation*}
This indicates that
\begin{equation*}
\partial_\omega\lambda^\alpha_l(\omega;u,\epsilon)={1}/{\varsigma}\left[
{\rm tr}(P(\omega)\partial_\omega(\Pi_\alpha(\omega)A_\alpha\Pi_\alpha(\omega))P(\omega))\right].
\end{equation*}
From formula \eqref{E5.1.14}, it yields that
\begin{equation*}
\partial_\omega\lambda^\alpha_l(\omega;u,\epsilon){=}
{1}/{\varsigma} [ {\rm tr}\left(P(\omega)\partial_\omega(\Pi^0_\alpha L^\alpha_\omega\Pi^0_\alpha)P(\omega) \right)]+O(\epsilon).
\end{equation*}
Moreover, the definition of $L^\alpha_\omega$ derives that $\Pi^0_\alpha L^\alpha_\omega\Pi^0_\alpha$ is diagonal with entries $-\omega^2j^2+{|n|}^4+m$, where $n\in\Omega_\alpha$, $j\in\mathbb{Z}$ with $K^{-1}_0\langle n(\alpha)\rangle^2\leq |j|\leq K_0\langle n(\alpha)\rangle^2$. This leads to
\begin{equation*}
-4K^2_0\langle n(\alpha)\rangle^4-C\epsilon\leq \partial_ {\omega} \lambda^\alpha_l(\omega;u,\epsilon)\leq -2K^{-2}_0\langle n(\alpha)\rangle^4+C\epsilon.
\end{equation*}
Consequently, we get \eqref{E5.1.6} if $\epsilon$ is in $(0,\gamma_0]$ with $\gamma_0$ small enough.
\\
{\rm(iii)} It is clear to read the first inequality in \eqref{E5.1.8} using \eqref{E5.1.7}. Since $\|\partial_\omega A_\alpha(\omega;u,\epsilon)\|_{\mathcal{L}(\mathcal{H}_0)}\leq C\langle n(\alpha)\rangle^{4}$ (see \eqref{E5.1.6}), by \eqref{E5.1.7}, we get the second inequality in \eqref{E5.1.8}.
\end{proof}
\subsection{Iterative scheme}
In this subsection, our goal is to achieve the proof of Theorem \ref{theorem2.1.1}. Fix indices $s,\sigma,N,\zeta,r,\delta$ satisfying the following inequalities
\begin{equation}\label{E5.2.1}
\sigma\geq\sigma_0+2(N+1)+(d+1)/2+{r}~\text{with} ~r=\zeta,~
(N+1)\rho'\geq r+2, \quad ~s\geq \sigma+\zeta+2, \quad ~\delta\in (0,\delta_0],
\end{equation}
where $\delta_0>0$ is small enough.
Let $m>0$ and the force term $f$ in \eqref{E3.3.17} be given in $\mathcal{H}^{s+\zeta}$.
Firstly we will solve Eq. \eqref{E3.3.17}. Our main task is to construct a sequence $(G_k,\mathcal{O}_k,\psi_k,u_k,w_k),k\geq 0$,
where $G_k,\mathcal{O}_k$ will be subsets of $[1,2]\times [0,\delta^2]$, $\psi_k$ will be a function of $(\omega,\epsilon)\in[1,2]\times [0,\delta^2]$, $u_k,w_k$
will be functions of $(t,x,\omega,\epsilon)\in \mathbb{T}\times \mathbb{T}^d\times [1,2]\times [0,\delta^2]$. When $k = 0$, let us define
\begin{align*}\label{E5.2.2}
&u_0=w_0=0,\\
&\mathcal{O}_0=\left\{(\omega,\epsilon)\in[1,2]\times [0,\gamma_0];\quad~\exists \alpha \in \mathcal{A}~\text{ with}~1\leq \langle n(\alpha)\rangle<2, ~
\exists~l\in\{1,\cdots,D_\alpha \}~\text{ with}~|\lambda^\alpha_l(\omega;0,\epsilon)|<2\delta \right\},\\
&G_0=\left\{(\omega,\epsilon)\in[1,2]\times [0,\gamma_0],~{\rm dist}(\omega,\mathbb{R}-\mathcal{O}_{0,\epsilon})\geq \frac{\delta}{128C_0}\right\},
\end{align*}
where $C_0$ is given in \eqref{E5.1.6}. For all $\epsilon\in[0,\gamma_0]$, $\mathcal{O}_{0,\epsilon},G_{0,\epsilon}$ denote the $\epsilon$-sections of $\mathcal{O}_0,G_0$ respectively. Obviously, $G_{0,\epsilon}$ is a closed subset of $[1,2]$ for all $\epsilon\in[0,\gamma_0]$, contained in the open subset $\mathcal{O}_{0,\epsilon}$.
By Urysohn's lemma, when $\epsilon$ is fixed, we may construct a $C^1$ function $\omega\rightarrow \psi_0(\omega,\epsilon)$,
compactly supported in $\mathcal{O}_{0,\epsilon}$, equal to 1 on $G_{0,\epsilon}$, satisfying for all $\omega,\epsilon$
\begin{equation*}
0\leq\psi_0(\omega,\epsilon)\leq 1, \quad ~|\partial_\omega\psi_0(\omega,\epsilon)|\leq C_1\delta^{-1},
\end{equation*}
where $C_1$ is some uniform constant depending only on $C_0$.

\begin{prop}\label{proposition5.2.1}
There are $\delta_0\in(0,\sqrt{\gamma_0}]$ with $\gamma_0$ small enough, positive constants $C_1,B_1,B_2$, a 5-uple $(G_k,\mathcal{O}_k,\psi_k,u_k,w_k)$ with, for all $k\geq0$, $\delta\in(0,\delta_0)$,
\begin{align}
\mathcal{O}_k
=\bigg\{(\omega,\epsilon)
&\in[1,2]\times [0,\delta^2]; \quad \exists \alpha\in \mathcal{A}~ \text {with}~
2^{k}\leq \langle n(\alpha)\rangle<2^{k+1}, \nonumber\\
& \exists  l\in \{1,\cdots,D_\alpha\},|\lambda^\alpha_l(\omega;u_{k-1},
\epsilon)|<2\delta 2^{-2k\zeta}\bigg\}, \nonumber\\
G_k=\bigg\{(\omega,\epsilon)
&\in[1,2]\times [0,\delta^2]; \quad {\rm dist}(\omega,\mathbb{R}-\mathcal{O}_{k,\epsilon})\geq
\frac{\delta}{128C_0}2^{-2k(\zeta+2)}\bigg\},\label{E5.2.4}
\end{align}
where $C_0$ is given in \eqref{E5.1.6}. And
\begin{align}
\psi_k:&
[1,2]\times [0,\delta^2]\rightarrow [0,1]~
\text {is~supported~in}~\mathcal{O}_k,{~equal~to~}1~ \text{ on}~G_k,
\nonumber\\
&C^1~ \text {in}~\omega \text{~and~} |\partial_\omega\psi_k(\omega,\epsilon)|\leq \frac{C_1}{\delta}2^{2k(\zeta+2)}~\forall(\omega,\epsilon)\in[1,2]\times [0,\delta^2].\label{E5.2.5}
\end{align}
For all $\epsilon\in[0,\delta^2]$, it can be showed that
\begin{equation*}
w_k\in\mathcal{H}^s, \quad  \partial_\omega w_k\in\mathcal{H}^{s-\zeta-2},
\end{equation*}
and $w_k(t,x,\omega,\epsilon)$, $\partial_\omega w_k(t,x,\omega,\epsilon)$ are continuous with respect to $\omega$  and satisfy
\begin{align}
&\|w_k(\cdot,\omega,\epsilon)\|_{\mathcal{H}^{s}}+\delta\|\partial_\omega w_k(\cdot,\omega,\epsilon)\|_{\mathcal{H}^{s-\zeta-2}}\leq B_1\frac{\epsilon}{\delta},\label{E5.2.6}\\
&\|w_k-w_{k-1}\|_{\mathcal{H}^\sigma}\leq B_2\frac{\epsilon}{\delta}2^{-2k\zeta} \label{E5.2.10}
\end{align}
uniformly in $\epsilon\in[0,\delta^2],\omega\in[1,2],\delta\in(0,\delta_0]$.
In addition, for all $(\omega,\epsilon)\in[1,2]\times[0,\delta^2]-\cup^k_{k'=0}\mathcal{O}_{k'}$, $w_k$ solves the following equation
\begin{align}
(L_\omega+\epsilon V_{\rm D}(u_{k-1},\omega,\epsilon))w_k=&\epsilon\tilde{S}_k{({\rm Id}+\epsilon Q(u_{k-1},\omega,\epsilon))^*}\tilde{R}(u_{k-1},\omega,\epsilon)u_{k-1}\nonumber\\
&+\epsilon\tilde{S}_k(R_1(u_{k-1},\omega,\epsilon)w_{k-1})+\epsilon\tilde{S}_k{({\rm Id}+\epsilon Q(u_{k-1},\omega,\epsilon))^*}f,\label{E5.2.7}
\end{align}
where
\begin{equation}\label{E5.2.3}
\tilde{S}_k=\sum \limits_{\alpha\in\mathcal{A};\langle n(\alpha)\rangle <2^{k+1}} \widetilde{\Pi}_\alpha,k\geq0,
\end{equation}
$\tilde{R}$ is defined in \eqref{E3.3.17} and $Q$, $V_{\rm D}$, $R_1$ are defined in \eqref{E4.2.2} and \eqref{E4.2.3}. The function $u_k$ is deduced from $w_k$ by
\begin{equation}\label{E5.2.8}
u_k(t,x,\omega,\epsilon)=({\rm Id}+\epsilon Q(u_{k-1},\omega,\epsilon))w_k ~\text{with}~
\end{equation}
\begin{align}
&\|u_k(\cdot,\omega,\epsilon)\|_{\mathcal{H}^{s}}+\delta\|\partial_\omega u_k(\cdot,\omega,\epsilon)\|_{\mathcal{H}^{s-\zeta-2}}\leq B_2\frac{\epsilon}{\delta},\label{E203}\\
&\|u_k-u_{k-1}\|_{\mathcal{H}^\sigma}\leq2B_2\frac{\epsilon}{\delta}2^{-2k\zeta}\label{E5.2.9}
\end{align}
uniformly for $\epsilon\in[0,\delta^2],\omega\in[1,2],\delta\in(0,\delta_0]$.
\end{prop}

\begin{rema}
For all $\delta_0\in(0,\sqrt{\gamma_0}]$ with $\gamma_0$ small enough, if we assume $\epsilon\leq\delta^2$, where $\delta\in(0,\delta_0)$, then \eqref{E203} implies that
\begin{equation}\label{E5.2.11}
\|u_k\|_{\mathcal{E}^\sigma(\zeta)}<q
\end{equation}
for some constant $q>0$.
\end{rema}
Before the proof of Proposition \ref{proposition5.2.1}, we need to introduce two lemmas.
\begin{lemm}\label{lemma5.2.1}
There is $\delta_0\in(0,1]$ small enough, depending only on the constants $B_1,B_2$, such that for all $k\geq0$, $k'\in\{0,...,k+1\}$, $\epsilon\in[0,\delta^2]$, $\delta\in(0,\delta_0]$, $\alpha\in\mathcal{A}$ with $2^{k'}\leq\langle n(\alpha)\rangle<2^{k'+1}$
\begin{equation*}
[1,2]-G_{k',\epsilon}\subset I(\alpha,u_k,\epsilon,\delta),
\end{equation*}
where $I(\cdot)$ is defined by \eqref{E5.1.7}. When $k=0$, we set $u_{-1}=0$.
\end{lemm}
\begin{proof}
We first consider $\omega\in[1,2]-\mathcal{O}_{k',\epsilon}$,
$l\in\{1,...,D_\alpha\}$. Owing to Proposition \ref{proposition5.1.1} {\rm(ii)}, \eqref{E5.2.4} and \eqref{E5.2.9},
setting $(u,u')=(u_k,u_{k'-1})$, there exists $l'\in\{1,...,D_\alpha\}$ such that
\begin{align}
|\lambda^\alpha_l(\omega;u_k,\epsilon)|&\geq|\lambda^\alpha_{l'}(\omega;u_{k'-1},\epsilon)|
-C_0\epsilon\|u_k-u_{k'-1}\|_{\mathcal{H}^\sigma}\nonumber\\
&\geq2\delta2^{-2'\zeta}-2C_0B_2\frac{\epsilon^2}{\delta}\frac{2^{-2k'\zeta}}{1-2^{-2\zeta}}
\geq\frac{3}{2}\delta2^{-2k'\zeta}\geq\frac{3}{2}\delta\langle n(\alpha)\rangle^{-2\zeta},\label{E5.2.14}
\end{align}
when ${\epsilon\leq\delta^2}$ if $\delta\in[0,\delta_0]$ with $\delta_0$ small enough.
Next, let $\omega\in\mathcal{O}_{k',\epsilon}-G_{k',\epsilon}$. The definition on $G_k$ (see\eqref{E5.2.4}) indicates that
\begin{equation}\label{E5.2.13}
 |\omega-\tilde{\omega}|<\frac{\delta}{128C_0}2^{-2k'(\zeta+2)},
\end{equation}
  where $\tilde{\omega}\in[1,2]-\mathcal{O}_{k',\epsilon}$.  By means of formula \eqref{E5.1.6},
  we obtain that for all $u\in\mathcal{E}^\sigma(\zeta)$ with $\|u\|_{\mathcal{E}^\sigma(\zeta)}<q$, $\alpha\in\mathcal{A}$, $l\in\{1,...,D_\alpha\}$
\begin{equation*}
\sup_{\omega\in[1,2]}|\partial_\omega\lambda^\alpha_l(\omega';u,\epsilon)|\leq 4C_0\langle n(\alpha)\rangle^4.
\end{equation*}
It follows from formulae \eqref{E5.2.14}, \eqref{E5.2.13} and the fact of $2^{k'}\leq\langle n(\alpha)\rangle<2^{k'+1}$ that
\begin{equation*}
|\lambda^\alpha_l(\omega;u_k,\epsilon)|\geq|\lambda^\alpha_l(\tilde{\omega};u_k,\epsilon)|-4C_0\langle n(\alpha)\rangle^4|\omega-\tilde{\omega}|{\geq}\delta2^{-2k'\zeta}\geq\delta\langle n(\alpha)\rangle^{-2\zeta}.
\end{equation*}
\end{proof}
In order to using the recurrence method, we shall also need to give the upper bound of the right-hand side  of equation \eqref{E5.2.7} at  $k+1$-th step. Denote
\begin{align}\label{E5.2.16}
H_{k+1}(u_k,w_k)=
&\tilde{S}_{k+1}~{({\rm Id}+\epsilon Q(u_{k},\omega,\epsilon))^*}\tilde{R}(u_k,\omega,\epsilon)u_{k}\nonumber\\
&+\tilde{S}_{k+1}(R_1(u_{k},\omega,\epsilon)w_k)
+\epsilon\tilde{S}_{k+1}{({\rm Id}+\epsilon Q(u_{k},\omega,\epsilon))^*}f.
\end{align}
\begin{lemm}\label{lemma5.2.2}
There exists  a constant $C > 0$, depending on $q$ in \eqref{E5.2.11} but independent of $k$,
such that for any $\omega\in[1, 2]$, any $\epsilon\in[0,\delta^2]$, any $\delta\in[0,\delta_0]$, the following holds:
\begin{align}
\|H_{k+1}(u_k,w_k)\|_{\mathcal{H}^{s+\zeta}}\leq
C(\|u_k&(\cdot,\omega,\epsilon)\|_{\mathcal{H}^{s}}+
\|w_k(\cdot,\omega,\epsilon)\|_{\mathcal{H}^{s}})
+(1+C\epsilon)\|f\|_{\mathcal{H}^{s+\zeta}}, \label{E5.2.17}\\
\|\partial_\omega H_{k+1}(u_k,w_k)\|_{\mathcal{H}^{s-2}}\leq
C(\|u_k&(\cdot,\omega,\epsilon)\|_{\mathcal{H}^{s}}
+\|\partial_\omega u_k(\cdot,\omega,\epsilon)\|_{\mathcal{H}^{s-\zeta-2}}\nonumber\\
+\|w_k&(\cdot,\omega,\epsilon)\|_{\mathcal{H}^{s}}
+\|\partial_\omega w_k(\cdot,\omega,\epsilon)\|_{\mathcal{H}^{s-\zeta-2}}
+\epsilon\|f\|_{\mathcal{H}^{s-2}}), \label{E5.2.18}\\
\|H_{k+1}(u_k,w_k)-H_{k}(u_{k-1}
,w_{k-1})
\|_{\mathcal{H}^{\sigma+\zeta}}
\leq&
C(\|u_k-u_{k-1}\|_{\mathcal{H}^{\sigma}}+\|w_k-w_{k-1}\|_{\mathcal{H}^{\sigma}})\nonumber\\
&+2^{-2(k+1)\zeta}C\left(\|u_k\|_{\mathcal{H}^{\sigma+\zeta}}+\|w_k\|_{\mathcal{H}^{\sigma+\zeta}}\right)\nonumber\\
&+2^{-2(k+1)\zeta}(1+C\epsilon)\|f\|_{\mathcal{H}^{\sigma+2\zeta}}. \label{E5.2.19}
\end{align}
\end{lemm}
\begin{proof}
Let $u_k$ satisfy \eqref{E5.2.11}. It follows from
Definition \ref{definition3.1.2} and \eqref{E5.2.1} that $\tilde{R},R_1$ are bounded from $\mathcal{H}^{s}$ to $\mathcal{H}^{s+\zeta}$ with
$s\in\mathbb{R}$. Moreover, Lemma \ref{lemma3.1.1} shows that ${Q(u_k,\omega,\epsilon)^*}$ is bounded on space $\mathcal{H}^{s}$ with $s\in\mathbb{R}$,  which yields \eqref{E5.2.17}.

The term in \eqref{E5.2.16} implies that we have to give  the upper bound of the following terms
\begin{subequations}
\begin{numcases}{}
\label{E5.2.20a}
\partial_\omega\big(Q(u_{k},\omega,\epsilon)\big)=\partial_uQ(\cdot,\omega,\epsilon)
\cdot(\partial_\omega u_k)+\partial_\omega Q(u_{k},\omega,\epsilon),\\
\label{E5.2.20b}
\partial_\omega\big(\tilde{R}(u_{k},\omega,\epsilon)\big)=\partial_u\tilde{R}(\cdot,\omega,\epsilon)
\cdot(\partial_\omega u_k)+\partial_\omega \tilde{R}(u_{k},\omega,\epsilon),\\
\label{E5.2.20c}
\partial_\omega\big(R_1(u_{k},\omega,\epsilon)\big)=\partial_uR_1(\cdot,\omega,\epsilon)
\cdot(\partial_\omega u_k)+\partial_\omega R_1(u_{k},\omega,\epsilon).
\end{numcases}
\end{subequations}
 The assumption on $s$ in \eqref{E5.2.1} shows $\mathcal{H}^{s-\zeta-2}\subset\mathcal{H}^\sigma$. Formulae \eqref{E3.1.2} and \eqref{E5.2.11} read that \eqref{E5.2.20a} is bounded on any space $\mathcal{H}^{s}$. Similarly, we see also that \eqref{E5.2.20b}, \eqref{E5.2.20c} are bounded from $\mathcal{H}^s$ to $\mathcal{H}^{s+\zeta}$.
This completes the proof of \eqref{E5.2.18}.

Let us write the difference of  $H_{k+1}(u_k,w_k)-H_{k}(u_{k-1},w_{k-1})$ as the sum of the following three parts:
\begin{align}\label{E5.2.21}
\begin{cases}
(\tilde{S}_{k+1}-\tilde{S}_{k}){({\rm Id}+\epsilon Q(u_{k},\omega,\epsilon))^*}\tilde{R}(u_k,\omega,\epsilon)u_{k},\\
(\tilde{S}_{k+1}-\tilde{S}_{k})R_1(u_{k},\omega,\epsilon)w_k,\\
(\tilde{S}_{k+1}-\tilde{S}_{k})~{({\rm Id}+\epsilon Q(u_{k},\omega,\epsilon))^*}{f},
\end{cases}
\end{align}
and
\begin{align}\label{E5.2.22}
\begin{cases}
\epsilon\tilde{S}_{k}({Q(u_{k},\omega,\epsilon)^*}-{Q(u_{k-1},\omega,\epsilon)^*})
\tilde{R}(u_k,\omega,\epsilon)u_{k},\\
\tilde{S}_{k}{({\rm Id}+\epsilon Q(u_{k-1},\omega,\epsilon))^*}(\tilde{R}(u_{k},\omega,\epsilon)-\tilde{R}(u_{k-1},\omega,\epsilon))u_k,\\
\tilde{S}_{k}(R_1(u_{k},\omega,\epsilon)-R_1(u_{k-1},\omega,\epsilon))w_k,\\
\epsilon\tilde{S}_{k}({Q(u_{k},\omega,\epsilon)^*}-{Q(u_{k-1},\omega,\epsilon)^*}){f},
\end{cases}
\end{align}
and
\begin{align}\label{E5.2.23}
\begin{cases}
\tilde{S}_{k}{({\rm Id}+\epsilon Q(u_{k-1},\omega,\epsilon))^*}\tilde{R}(u_{k-1},\omega,\epsilon)(u_{k}-u_{k-1}),\\
\tilde{S}_{k}R_1(u_k,\omega,\epsilon)(w_{k}-w_{k-1}).
\end{cases}
\end{align}
Formulae \eqref{E5.2.6} and \eqref{E203} lead to that $u_k,w_k$ are in a bounded subset of $\mathcal{H}^\sigma$.
This establishes that $\tilde{R}, R_1$ are bounded operators from $\mathcal{H}^{\sigma+\zeta}$ to $\mathcal{H}^{\sigma+2\zeta}$ with $\sigma\in\mathbb{R}$. Owing to \eqref{E5.2.3}, we establish that the $\mathcal{H}^{\sigma+\zeta}$-norm of \eqref{E5.2.21} is bounded from above by
\begin{equation*}
2^{-2(k+1)\zeta}C(\|u_k\|_{\mathcal{H}^{\sigma+\zeta}}+\|w_k\|_{\mathcal{H}^{\sigma+\zeta}})+2^{-2(k+1)\zeta}(1+C\epsilon)\|f\|_{\mathcal{H}^{\sigma+2\zeta}}.
\end{equation*}
 It follows from  \eqref{E3.1.3} and \eqref{E3.1.2} that there exists a constant $C$ such that
\begin{align*}
&\|\tilde{R}(u_k,\omega,\epsilon)-\tilde{R}(u_{k-1},\omega,\epsilon)\|_{\mathcal{L}(\mathcal{H}^\sigma,\mathcal{H}^{\sigma+\zeta})}\leq C\|u_k-u_{k-1}\|_{\mathcal{H}^\sigma},\\  &\|R_1(u_k,\omega,\epsilon)-R_1(u_{k-1},\omega,\epsilon)\|_{\mathcal{L}(\mathcal{H}^\sigma,\mathcal{H}^{\sigma+\zeta})}\leq C\|u_k-u_{k-1}\|_{\mathcal{H}^\sigma},\\
&\|{Q(u_k,\omega,\epsilon)^*}-{Q(u_{k-1},\omega,\epsilon)^*}\|_{\mathcal{L}(\mathcal{H}^{\sigma+\zeta},\mathcal{H}^{\sigma+\zeta})}\leq C\|u_k-u_{k-1}\|_{\mathcal{H}^\sigma}.
\end{align*}
Since ${Q(u_k,\omega,\epsilon)^*}$ is bounded on any space $\mathcal{H}^{\sigma}$ with $\sigma \in\mathbb{R}$, the $\mathcal{H}^{\sigma+\zeta}$-norm of \eqref{E5.2.22} is bounded from above by $C\|u_k-u_{k-1}\|_{\mathcal{H}^{\sigma}}+2^{-2(k+1)\zeta}C\epsilon \|f\|_{\mathcal{H}^{\sigma+2\zeta}}$.
It is easy to show that $\mathcal{H}^{\sigma+\zeta}$-norm of \eqref{E5.2.23} is bounded from above by $C(\|u_k-u_{k-1}\|_{\mathcal{H}^{\sigma}}+\|w_k-w_{k-1}\|_{\mathcal{H}^{\sigma}})$. Thus we get \eqref{E5.2.19}.
\end{proof}
Let us complete the proof of Proposition \ref{proposition5.2.1}.
\begin{proof}
We apply a recursive argument to  Proposition \ref{proposition5.2.1}.
We have already defined $(G_0,\mathcal{O}_0,\psi_0,u_0,w_0)$ satisfying \eqref{E5.2.4}-\eqref{E5.2.9}.
Suppose that $(G_k,\mathcal{O}_k,\psi_k,u_k,w_k)$ have been constructed satisfying \eqref{E5.2.4}- \eqref{E5.2.9}.
Now let us construct these datum at  $k + 1$-th step and verify that these datum  at $k + 1$-th step still satisfy \eqref{E5.2.4}-\eqref{E5.2.9}. When $u_k$ is given, the sets $\mathcal{O}_{k+1},G_{k+1}$ are defined by \eqref{E5.2.4} at $k+1$-th step. For fixed $\epsilon$, $G_{k+1,\epsilon}$ is a compact subset of the open set $\mathcal{O}_{k+1,\epsilon}$, where $G_{k+1,\epsilon}$ has to satisfy the distance between $G_{k+1,\epsilon}$ and the complement of $\mathcal{O}_{k+1,\epsilon}$ is bounded from below $\frac{\delta}{128C_0}2^{-2(k+1)(\zeta+2)}$. It is easy to  construct a function $\psi_{k+1}$ satisfying \eqref{E5.2.5} at the  $k+1$-th step applying Urysohn's lemma.

For $(\omega,\epsilon)\in[1,2]\times[0,\delta^2]-\cup^{k+1}_{k'=0}G_{k'}$,
let us construct $w_{k+1}$.
By construction, the operator $V_{\rm D}(u_k,\omega,\epsilon)$ is  a block-diagonal operator, which implies
\begin{equation*}
({L}_\omega+\epsilon V_{\rm D}(u_{k},\omega,\epsilon))\widetilde{{\Pi}}_\alpha w_{k+1}=\widetilde{{\Pi}}_\alpha({L}_\omega+\epsilon V_{\rm D}(u_{k},\omega,\epsilon))w_{k+1}.
\end{equation*}
Combining this with formula \eqref{E5.2.16}, we obtain that Eq. \eqref{E5.2.7} at the  $k+1$-th step can be written as
\begin{equation}\label{E5.2.24}
({L}_\omega+\epsilon V_{\rm D}(u_{k},\omega,\epsilon))\widetilde{{\Pi}}_\alpha w_{k+1}=\epsilon\widetilde{{\Pi}}_\alpha H_{k+1}(u_k,w_k),\alpha\in\mathcal{A}.
\end{equation}
Remark that  the right-hand side of \eqref{E5.2.24} vanishes when $\langle n(\alpha)\rangle\geq2^{k+2}$ by \eqref{E5.2.3}.
Let $k'\in\{0,...,k+1\}$, $n\in \mathbb{N}$ with
$2^{k'}\leq\langle n(\alpha)\rangle<2^{k'+1}$, $\omega\in[1,2]-G_{k',\epsilon}$.
It follows from Lemma \ref{lemma5.2.1}, Proposition \ref{proposition5.1.1} {\rm (iii)} that Eq. \eqref{E5.2.24} may be simplified as
\begin{equation}\label{E5.2.25}
\widetilde{{\Pi}}_\alpha w_{k+1}=\epsilon A_\alpha(\omega;u_k,\epsilon)^{-1}\widetilde{{\Pi}}_\alpha H_{k+1}(u_k,w_k).
\end{equation}
Then $w_{k+1}$ is defined as
\begin{equation}\label{E5.2.28}
w_{k+1}(t,x,\omega,\epsilon)=:\sum_{k'=0}^{k+1}\sum\limits_{{\alpha\in\mathcal{A}},{2^{k'}\leq\langle n(\alpha)\rangle<2^{k'+1}}}(
1-\psi_{k'}(\omega,\epsilon))\widetilde{\Pi}_\alpha w_{k+1}(t,x,\omega,\epsilon)
\end{equation}
for all  $(\omega,\epsilon)\in[1,2]\times[0,\delta^2]$. Let us firstly verify that \eqref{E5.2.6} holds at the  $k+1$-th step.
Formulae \eqref{E5.1.8} and \eqref{E5.2.25} deduce that for all $k'\in\{0,...,k+1\}$,
$\alpha\in \mathcal{A}$ with $2^{k'}\leq\langle n(\alpha)\rangle<2^{k'+1}$, $(\omega,\epsilon)\in[1,2]\times[0,\delta^2]-G_{k'}$
\begin{equation}\label{E5.2.26}
\|\widetilde{{\Pi}}_\alpha w_{k+1}(\cdot,\omega,\epsilon)\|_{\mathcal{H}^s}\leq E_0\frac{\epsilon}{\delta}\|\widetilde{{\Pi}}_\alpha H_{k+1}(u_k,w_k)(\cdot,\omega,\epsilon)\|
_{\mathcal{H}^{s+\zeta}},
\end{equation}
and
\begin{align}\label{E5.2.27}
\|\widetilde{{\Pi}}_\alpha\partial_\omega w_{k+1}(\cdot,\omega,\epsilon)\|_{\mathcal{H}^{s-\zeta-2}}\leq&
E_0\frac{\epsilon}{\delta}\|\widetilde{{\Pi}}_\alpha\partial_\omega H_{k+1}(u_k,w_k)(\cdot,\omega,\epsilon)\|
_{\mathcal{H}^{s-2}}\nonumber\\
&+E_0\frac{\epsilon}{\delta^2}\|\widetilde{{\Pi}}_\alpha H_{k+1}(u_k,w_k)(\cdot,\omega,\epsilon)\|
_{\mathcal{H}^{s+\zeta}}.
\end{align}
Furthermore formula \eqref{E5.2.5} gives that
\begin{equation}\label{E6.41}
\|\partial_\omega\psi_{k'}\widetilde{{\Pi}}_\alpha w_{k+1}\|_{\mathcal{H}^{s-\zeta-2}}
\leq\frac{C_1}{\delta}\|\widetilde{{\Pi}}_\alpha w_{k+1}\|_{\mathcal{H}^s}.
\end{equation}
From \eqref{E5.2.1}, \eqref{E5.2.11}, \eqref{E5.2.6},  \eqref{E203}, \eqref{E5.2.17}, \eqref{E5.2.18}, and \eqref{E5.2.28}-
\eqref{E6.41}, it yields that
\begin{align*}
\|w_{k+1}(\cdot,\omega,\epsilon)\|_{\mathcal{H}^s}\leq
&E_0\frac{\epsilon}{\delta}\left(C\frac{\epsilon}{\delta}(B_1+B_2)+(1+C\epsilon)
\|f\|_{\mathcal{H}^{s+\zeta}}\right),\\
\|\partial_\omega w_{k+1}(\cdot,\omega,\epsilon)\|_{\mathcal{H}^{s-\zeta-2}}\leq
&
 E_0\frac{\epsilon}{\delta}\left(C\frac{\epsilon}{\delta^2}(B_1+B_2)+C\epsilon\|f\|
_{\mathcal{H}^{s-2}}\right)\\
&+E_0\frac{\epsilon}{\delta^2}\left(C\frac{\epsilon}{\delta}(B_1+B_2)
+(1+C\epsilon)\|f\|
_{\mathcal{H}^{s+\zeta}}\right)\\
&+E_0C_1\frac{\epsilon}{\delta^2}\left(\frac{\epsilon}{\delta}C(B_1+B_2)+(1+C\epsilon)\|f\|
_{\mathcal{H}^{s+\zeta}}\right).
\end{align*}
Notice that $C$ depends only on $q,E_0, C_1$, where $q$ is given by \eqref{E5.2.11},
$E_0, C_1$ are uniform constants. If $\epsilon\leq\delta^2\leq\delta^2_0$ with $\delta_0$ small enough,
when $B_1$ is taken large enough corresponding  to $E_0,C_1$ and $\|f\|_{\mathcal{H}^{s+\zeta}}$, then
we have that \eqref{E5.2.6} still holds at the  $k + 1$-th step.
Furthermore, using that ${Q(u_k,\omega,\epsilon)}$ is bounded on any space $\mathcal{H}^{s}$ with $s\in\mathbb{R}$, we derive
\begin{align*}
\|u_{k+1}(\cdot,\omega,\epsilon)\|_{\mathcal{H}^{s}}&+\delta\|\partial_\omega u_{k+1}(\cdot,\omega,\epsilon)\|_{\mathcal{H}^{s-\zeta-2}}
\leq(1+C\epsilon+C\epsilon\delta)
(\|w_{k}\|_{\mathcal{H}^{s}}+\delta\|\partial_\omega w_{k+1}\|_{\mathcal{H}^{s-\zeta-2}}).
\end{align*}
When $\delta_0$ is small enough, if we take $B_2 = 2B_1$, then \eqref{E203} holds the  $k + 1$-th step.

Next, we check \eqref{E5.2.10} still holds  at the  $k + 1$-th step. It is straightforward to obtain that
 \begin{align}
(W_{k+1}-W_{k})(t,x,\omega,\epsilon)=&\sum\limits_{k'=0}^k\sum\limits_{\alpha\in\mathcal{A},2^{k'}\leq\langle n(\alpha)\rangle<2^{k'+1}}
(1-\varphi_{k'})\widetilde{\Pi}_\alpha(W_{k+1}-W_{k})\nonumber\\
&+\sum\limits_{\alpha\in\mathcal{A},2^{k+1}\leq\langle n(\alpha)\rangle<2^{k+2}}
(1-\varphi_{k+1})\widetilde{\Pi}_\alpha W_{k+1}\label{E5.2.20}
\end{align}
due to \eqref{E5.2.28}. For all $k'\in\{0,...,k+1\}$,  $(\omega,\epsilon)\in[1,2]\times[0,\delta^2]-G_{k'}$, $\alpha \in\mathcal{A}$ with $2^{k'}\leq\langle n(\alpha)\rangle<2^{k'+1}$, formulae \eqref{E5.2.25} and \eqref{E5.1.8} give the upper bound
\begin{align}
\|\widetilde{{\Pi}}_\alpha(w_{k+1}-w_{k})\|_{\mathcal{H}^\sigma}{\leq}
&E_0\frac{\epsilon}{\delta}(\|\widetilde{{\Pi}}_\alpha
(V_{\rm D}(u_{k-1},\omega,\epsilon)-V_{\rm D}(u_{k},\omega,\epsilon))w_k\|_{\mathcal{H}^{\sigma+\zeta}}\nonumber\\
&+\|\widetilde{{\Pi}}_\alpha(H_{k+1}(u_k,w_k)-H_{k}(u_{k-1},w_{k-1}))\|_{\mathcal{H}^{\sigma+\zeta}}). \label{E5.2.32}
\end{align}
Furthermore, by formula \eqref{E3.1.2}, it infers that for $s\geq\sigma+\zeta$
\begin{equation}\label{E6.47}
\|(V_{\rm D}(u_{k-1},\omega,\epsilon)-V_{\rm D}(u_{k},\omega,\epsilon))
w_k\|_{\mathcal{H}^{\sigma+\zeta}}\leq C\|u_k-u_{k-1}\|_{\mathcal{H}^{\sigma}}\|w_k\|_{\mathcal{H}^{s}}.
\end{equation}
Applying \eqref{E5.2.1} and \eqref{E5.2.6}, there exist some universal constants $C_3$ such that
\begin{equation}\label{E6.48}
\Big\|\sum\limits_{\alpha\in\mathcal{A},{2^{k+1}\leq\langle n(\alpha)\rangle<2^{k+2}}}
(1-\varphi_{k+1}){\Pi}_n w_{k+1}\Big\|_{\mathcal{H}^\sigma}\leq C_32^{-2(k+1)(s-\sigma)}\|w_{k+1}\|_{\mathcal{H}^s}
\leq C_3B_1\frac{\epsilon}{\delta}2^{-2((k+1)s-\sigma)}.
\end{equation}
Owing to  \eqref{E5.2.19},\eqref{E5.2.1}, \eqref{E5.2.11}, \eqref{E5.2.6},  \eqref{E5.2.10}, \eqref{E203}, \eqref{E5.2.9}, \eqref{E5.2.20}-
 \eqref{E6.48}, it yields that for $s\geq\sigma+\zeta$
\begin{align*}
\|w_{k+1}-w_{k}\|_{\mathcal{H}^{\sigma}}
\leq& E_0\frac{\epsilon}{\delta}\big(2CB_1B_2\frac{\epsilon^2}{\delta^2}2^{-2(k+1)\zeta}
+3CB_2\frac{\epsilon}{\delta}2^{-2(k+1)\zeta}+C\frac{\epsilon}{\delta}(B_1+B_2)2^{-2(k+1)\zeta}\\
&+(1+C\epsilon)\|f\|_{\mathcal{H}^{\sigma+2\zeta}}2^{-2(k+1)\zeta}\big)
+C_2B_1\frac{\epsilon}{\delta}2^{-2(k+1)(s-\sigma)}.
\end{align*}
We have $\|f\|_{\mathcal{H}^{\sigma+2\zeta}}\leq \|f\|_{\mathcal{H}^{s+\zeta}}$
from $s\geq\sigma+\zeta$. If $0\leq\epsilon\leq\delta^2\leq\delta^2_0$ with $\delta_0$ small enough,
when $B_1$ is chosen large enough relatively to $E_0,\|f\|_{\mathcal{H}^{s+\zeta}}$,
and $B_2$ is taken large enough corresponding to $B_1,C_2$, then \eqref{E5.2.10} is obtained for $s\geq\sigma+\zeta$ holds the  $k + 1$-th step.
It is clear to verify that \eqref{E5.2.9} still holds at the  $k+1$-th step by the definition  \eqref{E5.2.8}.
 This concludes the proof of the proposition.
\end{proof}
Our aim is to construct the solution of \eqref{E3.3.17}. Therefore we consider the equation about  $u_k$.
According to \eqref{E5.2.8}, Proposition \ref{proposition4.2.2}, \eqref{E5.2.11} and \eqref{E5.2.7},
it follows that for any $(\omega,\epsilon)\in[1,2]\times[0,\delta^2]-\cup^k_{k'=0}\mathcal{O}_{k'}
,\delta\in(0,\delta_0]$
\begin{align}
({L}_\omega
&+\epsilon V(u_{k-1},\omega,\epsilon))u_k=
\epsilon({({\rm Id}+\epsilon Q(u_{k-1},\omega,\epsilon))^*})^{-1}\big(\tilde{S}_k{({\rm Id}+
\epsilon Q(u_{k-1},\omega,\epsilon))^*}\tilde{R}(u_{k-1},\omega,\epsilon)u_{k-1}\nonumber\\
&+\tilde{S}_k(R_1(u_{k-1},\omega,\epsilon)w_{k-1})
+(\tilde{S}_k{({\rm Id}+\epsilon Q(u_{k-1},\omega,\epsilon))^*}{f}+R_1(u_{k-1},\omega,\epsilon)w_k)\big).\label{E5.2.12}
\end{align}

Finally, let us complete the proof of Theorem \ref{theorem2.1.1}.
\begin{proof}
 Formulae \eqref{E5.2.8} and \eqref{E5.2.9} indicate that the sequence $u_k$ is well defined and converges to $u$ in $\mathcal{H}^{\sigma}$  with
\begin{equation*}
\|u(\cdot,\omega,\epsilon)\|_{\mathcal{H}^{s}}+\delta\|\partial_\omega u(\cdot,\omega,\epsilon)\|_{\mathcal{H}^{s-\zeta-2}}\leq B_2\frac{\epsilon}{\delta}.
\end{equation*}
Moreover, by \eqref{E5.2.6}, \eqref{E5.2.10}, the sequence $w_k$ converges in $\mathcal{H}^{\sigma}$ to $w$, which satisfies
\begin{equation*}
\|w(\cdot,\omega,\epsilon)\|_{\mathcal{H}^{s}}+\delta\|\partial_\omega w(\cdot,\omega,\epsilon)\|_{\mathcal{H}^{s-\zeta-2}}\leq B_1{\textstyle\frac{\epsilon}{\delta}}.
\end{equation*}
If $(\omega,\epsilon)$ is in $[1,2]\times[0,\delta^2]-\bigcup^{+\infty}_{k'=0}G_{k'}$,
$\delta\in(0,\delta_0]$ with $\delta_0$ small enough, then equation \eqref{E5.2.12} is satisfied for any $k\in \mathbb{N}$.
Therefore $u$ satisfies
\begin{equation*}
({L}_\omega+\epsilon V(u,\omega,\epsilon))u=\epsilon \tilde{R}(u,\omega,\epsilon)u+\epsilon {f}
\end{equation*}
as $k\rightarrow+\infty$. This shows that $u$ is a solution of Eq. \eqref{E3.3.17}. By Proposition \ref{proposition3.3.1}, Eq. \eqref{E3.3.17} is equivalent to Eq. \eqref{E3.3.4} which is also equivalent to \eqref{E3.2.11} by Proposition \ref{proposition3.2.1}. Thus we may get a solution satisfying the conditions of Theorem \ref{theorem2.1.1}. Let $\mathcal{O}=\bigcup^{+\infty}_{k'=0}\mathcal{O}_{k'}$.
For $\omega,\omega'\in\mathcal{O}_{k',\epsilon}$, using \eqref{E5.1.6} and \eqref{E5.2.4}, we may obtain the bound
\begin{equation*}
|\omega-\omega'|\stackrel{\theta\in(0,1)}{=}\frac{|\lambda_n^l(\omega;u_{k'},\epsilon)-\lambda_n^l(\omega';u_{k'},\epsilon)|}
{|\partial_\omega\lambda_n^l(\theta\omega'+(1-\theta)\omega;u_{k'},\epsilon)|}\leq C2^{-2(2+\zeta)k'}\delta.
\end{equation*}
Moreover, from  $\langle n(\alpha)\rangle<2^{k'+1}$ and definition of $\widetilde{\Pi}_\alpha$, we deduce that $D_\alpha\leq C_12^{(k'+1)(\beta d+2)}$ with $\alpha\in\Omega_{\alpha}$. Thus the upper bound of $\omega$-measure of the $\epsilon$-section of $\mathcal{O}$ is
\begin{equation*}
C\delta\sum_{k'=0}^{+\infty}2^{-2(2+\zeta)k'+2{(k'+1)(\beta d+2)}+(k'+1)d}.
\end{equation*}	
The series converges if we take $\zeta>\beta d+{d}/{2}+2$. This implies that we obtain the bound $O(\delta)$, which gives the proof of \eqref{E2.1.4}.	
\end{proof}


\end{document}